\documentclass[a4paper,11pt]{article}
\usepackage{hyperref}
\usepackage{geometry}
\usepackage[utf8]{inputenc}
\usepackage[T1]{fontenc}
\usepackage{lmodern}
\usepackage{amssymb}
\usepackage{amsmath}
\usepackage{amsthm}
\usepackage{amsfonts}
\usepackage{commath}
\usepackage{mathrsfs}
\usepackage{paralist}
\usepackage{mathtools}
\usepackage{cleveref}
\usepackage{tikz}
\usepackage{framed}
\usepackage{authblk}
\usepackage{titlefoot}
\usepackage{xcolor}
\usepackage{url}
\usepackage[numbers]{natbib}

\newtheorem{theorem}{Theorem}[section]
\newtheorem*{theorem*}{Theorem}
\newtheorem{remark}[theorem]{Remark}
\newtheorem{definition}[theorem]{Definition}
\newtheorem*{definition*}{Definition}

\newtheorem{lemma}[theorem]{Lemma}
\newtheorem*{lemma*}{Lemma}
\newtheorem{corollary}[theorem]{Corollary}
\newtheorem*{corollary*}{Corollary}
\newtheoremstyle{boldremark}
    {\dimexpr\topsep/2\relax} 
    {\dimexpr\topsep/2\relax} 
    {}          
    {}          
    {\bfseries} 
    {.}         
    {.5em}      
    {}          
\theoremstyle{boldremark}

\newtheorem{assumption}[theorem]{Assumption}
\newtheorem*{assumption*}{Assumption}

\newcommand{\R}{\mathbb{R}}
\newcommand{\N}{\mathbb{N}}

\newcommand{\W}{\mathbb{W}}
\renewcommand{\P}{\mathbb{P}}
\newcommand{\E}{\mathbb{E}}

\newcommand{\A}{\mathbb{A}}

\newcommand{\cF}{\mathcal{F}}

\newcommand{\cH}{\mathcal{H}}
\newcommand{\cN}{\mathcal{N}}

\newcommand{\cP}{\mathcal{P}}

\newcommand{\cS}{\mathcal{S}}

\newcommand{\cL}{\mathcal{L}}

\newcommand{\cY}{\mathcal{Y}}

\newcommand{\refforappendixone}{\ref{appendix:General Skorokhod and Application}}
\newcommand{\refforappendixtwo}{\ref{appendix:topologies on general function spaces}}

\begin{document}


\title{Pontryagin Maximum Principle for McKean-Vlasov Stochastic Reaction-Diffusion Equations}
\date{\today}
	
\author[1]{Johan Benedikt Spille}
\author[1]{Wilhelm Stannat}
	
\affil[1]{Technische Universität Berlin, Berlin, Germany}
	
	
\maketitle

\unmarkedfntext{\textit{Mathematics Subject Classification (2020) --- 93E20, 49K45, 49N80, 60H15, 35K57} }
	
\unmarkedfntext{\textit{Keywords and phrases --- Pontryagin maximum principle, McKean-Vlasov stochastic reaction-diffusion equation, backward McKean-Vlasov stochastic partial differential equations, Lions derivative}}
	
\unmarkedfntext{\textit{Mail}: \textbullet \href{mailto:spille@math.tu-berlin.de}{spille@math.tu-berlin.de},\,\textbullet \href{mailto:stannat@math.tu-berlin.de}{stannat@math.tu-berlin.de}}

\begin{abstract}
    We consider the stochastic control of a semi-linear stochastic partial differential equations (SPDE) of McKean-Vlasov type. 
    Based on a recent novel approach to the Lions derivative for Banach space valued functions, we prove the Gateaux differentiability of the control to state map and, using adjoint calculus, we derive explicit representations of the gradient of the cost functional and a Pontryagin maximum principle. On the way, we also prove a novel existence and uniqueness result for linear McKean-Vlasov backward SPDE.\\
    Furthermore, for deterministic controls, we prove the existence of optimal controls using a martingale approach and a novel compactness method. This result is complemented in the appendix with a rigorous proof of folklore results on the compactness method in the variational approach to SPDE.\\
    Our setting uses the variational approach to SPDE with monotone coefficients, allowing for a polynomial perturbation and allowing the drift and diffusion coefficients to depend on the state, the distribution of the state and the control.
\end{abstract}

\tableofcontents

\section{Introduction}
In this paper, we are concerned with the control of the McKean-Vlasov SPDE
\begin{equation}
    \begin{aligned}
        dX^\alpha_t& 
        =LX^\alpha_t+F\left(t, X^\alpha_t, \mathcal{L}\left(X^\alpha_t\right), \alpha_t\right) d t+B\left(t, X^\alpha_t, \mathcal{L}\left(X^\alpha_t\right), \alpha_t\right) d W_t,\\
        X_0& \in L^2\left(\Omega, \mathcal{F}_0, \mathbb{P} ; H\right),
    \end{aligned} 
    \label{eq:TheBase-State-SPDE}
\end{equation}
on some separable Hilbert space $H$, where $\cL(X^\alpha_t)$ denotes the law of $X^\alpha_t$. We want to optimize with respect to the cost functional
\begin{equation}
    J: \mathbb{A} \rightarrow \mathbb{R}, \quad \alpha \mapsto \mathbb{E}\left[\int_0^T f\left(t, X^\alpha_t, \mathcal{L}\left(X^\alpha_t\right), \alpha_t\right) d t+g\left(X^\alpha_T, \mathcal{L}\left(X^\alpha_T\right)\right)\right], \label{eq:Cost_Functional}
\end{equation}
where the controls take values in the separable Hilbert space $U$ and fulfill the following integrability condition 
\begin{equation*}
    \alpha\in \mathbb{A}:=\left\{\alpha\in L^q(\Omega\times [0,T],U)\,\middle|\,\alpha\text{ progr. meas., }\int_0^T\left\| \alpha(t)\right\|_U^q d t \leq K \quad\P\text{ a.s.}\right\},
\end{equation*}
for some $K>0$ and some later specified $q\geq 2$.
Using a new extension of the Lions derivative to infinite dimensions from  \cite{vogler2024lionsderivativeinfinitedimensions}, we will develop, as our main result, a Pontryagin maximum principle in Theorem \ref{theorem:Pontryagin Maximum Principle}. On the way, as a secondary result in Theorem \ref{theorem:adjoint existence and uniqueness}, we prove existence and uniqueness for a type of backward McKean-Vlasov SPDE. Finally, as our second main result in Theorem \ref{theorem: existence of an optimal control}, we will show the existence of an optimal control, when the controls are deterministic. Precise assumptions on the coefficients will be given in section \ref{section:Preliminaries and Assumptions}. In particular, our results include the standard linear-quadratic case and allow for a semi-linear equation with polynomial perturbation.

The optimal control of SPDEs has been a topic of significant research over the past decades, motivated by applications in distributed systems, stochastic physics, and mathematical finance. When the underlying dynamics also include interactions through the law of the state, called McKean-Vlasov SPDEs, the complexity increases substantially. These equations describe the evolution of spatially distributed random fields influenced both by local randomness and by their own statistical distribution, and are thus suitable for modeling large-scale interacting systems under uncertainty.

Classical approaches to the stochastic optimal control of finite-dimensional equations are well developed, with the Pontryagin maximum principle providing necessary conditions for optimality \cite{Bismut1978,Peng1990}. Extensions to McKean-Vlasov stochastic differential equations have gained significant traction due to their relevance in mean-field games and collective behavior modeling \cite{LasryLions2007,CarmonaDelarue2018}. Optimal control in this context involves handling distributional derivatives, typically using the Lions derivative on the Wasserstein space of probability measures \cite{CarmonaDelarue2018,Buckdahn2014,PhamWei2020}.

Simultaneously, SPDEs have been rigorously studied in the context of infinite-dimensional stochastic analysis \cite{DaPrato2014,Liu2015,Pardoux2021}, and optimal control in this setting has been approached via both the dynamic programming principle and the maximum principle. The stochastic maximum principle in distributed parameter systems was originally developed in \cite{BENSOUSSAN1983}, laying the foundation for further developments in infinite-dimensional control. This variational approach has proven to be an effective tool to derive necessary conditions for optimality in the SPDE setting \cite{Fuhrman2013,Lu2021}. However, results that combine the McKean-Vlasov structure with SPDE dynamics remain limited, due to the analytical difficulties posed by the combination of infinite-dimensionality and nonlocal dependence on the distribution of the state.  \cite{Ahmed2016} used a functional derivative in a mild setting, to derive a maximum principle.  \cite{Dumitrescu2018}, similar to  \cite{Tang2019}, used a special form of the distribution dependence, to avoid explicitly needing the generalized Lions derivative.  \cite{CossoGozziFausto2023} have treated the dynamic programming approach in infinite dimension, though only for real-valued distribution dependent functions, due to the lack of a general Lions-derivative.

In this paper, we address these gaps by studying the optimal control of a class of semi-linear SPDEs of McKean-Vlasov type, where the coefficients may depend on the state, the control, and the law of the state. We employ a variational approach to SPDEs in the framework of monotone operators (cf. \cite{Liu2015}, \cite{Gao_2022}, \cite{HongHuLiu2022}). The key technical novelty is the use of a recent extension of the Lions derivative to Banach space valued functions \cite{vogler2024lionsderivativeinfinitedimensions}, which allows us to rigorously define the adjoint equation in our infinite-dimensional setting.

Our work generalizes two recent lines of research: the deterministic control of SPDEs with monotone coefficients developed in \cite{StannatWessels2021}, and the control of mean-field equations in finite dimensions developed in \cite{HocquetVogler2020}. We integrate these perspectives into a unified infinite-dimensional, mean-field control framework.

The paper is structured in the following way. 
In section \ref{section:Preliminaries and Assumptions} we introduce our setting, give a short introduction to the Lions derivative and list our assumptions needed for our results. 
In section \ref{section:Well-Posedness of the Control Problem} we show some preliminary results, that lead to the well-posedness of our problem. 
Then, in section \ref{section:Regularity of Control-to-State-Map and Cost Functional} we establish the Gâteaux differentiability of the control-to-state map and the cost functional. 
In section \ref{section:Adjoint Calculus and the Pontryagin Maximum Principle} we establish a, to our knowledge, novel existence and uniqueness result for linear backward McKean-Vlasov SPDE arising in the adjoint equation. Here, the coefficients are not Lipschitz uniformly in $\Omega$, in comparison to the results by \cite{Tang2019}. Using the solution to the adjoint equation, we derive an explicit expression for the gradient of the cost functional. This leads to a stochastic Pontryagin maximum principle for McKean-Vlasov SPDEs with distributional dependence in both drift and diffusion coefficients. 
In section \ref{section:existence of an optimal control}, we prove the existence of optimal controls in the special case, that the controls are deterministic. We do this, by adapting a compactness method from  \cite{Pardoux2021}, where it is used to prove existence of solutions to SPDE. Since this technique is attributed to an unpublished thesis by M.~Viot, we provide a complete proof of this method in the appendix for the readers convenience.
In the last section \ref{section:examples}, we show, that our results are applicable to the linear-quadratic case and give an example of a reaction-diffusion equation, where the results are applicable.
\section{Preliminaries and Assumptions}\label{section:Preliminaries and Assumptions}
    We fix a final time $T>0$ and let $H$ and $U$ be separable Hilbert spaces with inner products $\langle\cdot, \cdot\rangle_H$ and $\langle\cdot, \cdot\rangle_U$ and norms $\|\cdot\|_H$ and $\|\cdot\|_U$. Let $\left(\Omega, \mathcal{F},\{\mathcal{F}_t\}_{t>0}, \mathbb{P}\right)$ be an atomless filtered probability space, where $\cF_t$ is the augmented natural filtration generated by a cylindrical Wiener process $(W_t)$ on $H$. Further, let $V$ be a separable, reflexive Banach space with norm $\|\cdot\|_V$ such that $V \subset H \simeq H^* \subset V^*$ is a Gelfand triple. The dualization of $V^*$ and $V$ will be denoted $\langle\cdot,\cdot\rangle_V$. $\cP_2(H)$ will denote the $2$-Wasserstein space and $\W_2(\mu,\nu)=\inf_{\pi\in\Pi(\mu,\nu)} \int_{H} \|x-y\|d\pi(x,y)$ the Wasserstein-$2$-distance, where in $\Pi(\mu,\nu)$ are all the measures on $H\times H$ with marginals $\mu$ and $\nu$. Note, that here, for a bounded sequence of measures $(\mu_n)_{n\in\N}\subset \cP_2(H)$, it is equivalent, that $\mu_n$ converges weakly (in the probabilistic sense) to $\mu\in \cP_2(H)$ and $\W_2(\mu_n,\mu)\to 0$ (cf. \cite{Villani2009} Theorem 6.9).\\
    Our state equation \eqref{eq:TheBase-State-SPDE} will be seen as a variational SPDE with law dependence. We will later specify the precise notion of a solution in Definition \ref{def:Solution to SPDE}. Its coefficients are a linear, bounded operator $L:V\to V^*$ and
    \begin{equation*}
        F:[0, T] \times V \times \mathcal{P}_2(H) \times U\rightarrow V^* \quad \text{and}\quad
        B:[0, T] \times H \times \mathcal{P}_2(H) \times U\rightarrow L_2(H),
    \end{equation*}
    which for now are just measurable mappings, where $L_2(H)$ is the space of Hilbert-Schmidt operators on $H$. In the same vein, the coefficients of the cost functional \eqref{eq:Cost_Functional}
    \begin{equation*}
        f:[0, T] \times H \times \mathcal{P}_2\left(H\right) \times U \rightarrow \mathbb{R} \quad\text{and}\quad  g: H \times \mathcal{P}_2\left(H\right) \rightarrow \mathbb{R}
    \end{equation*}
    are for now measurable. Further properties will be specified later.
    For brevity sake, we will sometimes be using the shorthand notation $\theta_t=(t,X^\alpha_t,\cL(X^\alpha_t),\alpha_t)$ or versions of it, that should be specified locally. Finally, we will be using $A_i$ as constants from assumptions, $C$ as constants coming from a theorem, lemma etc. and $c_i$ as constants inside a proof, that are only uniquely determined inside this proof.
\subsection{The Lions Derivative in Infinite Dimensions} \label{subsection:Lions derivative}
    Since only a generalization of the Lions derivative to mappings with infinite dimensional co-domains from  \cite{vogler2024lionsderivativeinfinitedimensions} makes this paper possible, we will give a short introduction to the idea. For the finite dimensional case, we refer to  \cite{CarmonaDelarue2018} for a detailed introduction.\\
    Assume, we have a mapping $h:\cP_2(H)\to V^*$.
    Then, look at the so-called lift
    \begin{equation*}
        \hat{h}: L^2(\Omega,H)\to V^*, X\mapsto h(\cL(X))
    \end{equation*}
    and the Fréchet derivative $D\hat{h}(X)$ of this mapping. 
    The important fact here is, that it can be shown, that the derivative $D\hat{h}(X)$ does not depend on the chosen random variable $X$, but only on its distribution, or more precisely $D\hat{h}(X)Y$ only depends on $\cL(X,Y)$. If we had $V^*=\R$ now, we could identify $D\hat{h}(X)\in L(L^2(\Omega,\R),\R)=(L^2(\Omega, \R))^*$, so using Riesz' representation theorem, we get
    \begin{equation*}
        D\hat{h}(X)Y=\E\left[Z Y\right]
    \end{equation*}
    for some $Z\in L^2(\Omega,\R)$ and, by showing the measurability of $Z$ with respect to $\sigma(X)$, the factorization theorem gives us a function $\partial_\mu h(\mu) : \R\to\R$, which is then called the Lions derivative, such that
    \begin{equation*}
        D\hat{h}(X)Y=\E\left[\partial_\mu h(\mu)(X) Y\right].
    \end{equation*}
    This method does not work in infinite dimensions, as the Riesz representation theorem does not apply anymore. The key observation made by \cite{vogler2024lionsderivativeinfinitedimensions} is, that $\partial_\mu h(\mu)$ can also be seen, as a Radon-Nikodym derivative (or density) of $\mu_{D\hat{h}(X)}\circ X^{-1}$ with respect to $\mu=\cL(X)$, where 
    \begin{equation}
        \mu_{D\hat{h}(X)}(A)=D\hat{h}(X)\mathbf{1}_A. \label{eq:Relation vector measures and Linear Operators}
    \end{equation}
    If $V^*=\R$, this is a real-valued measure and it leads (by using the standard Radon-Nikodym theorem) to the same factorization as before. In infinite dimensions, though, we have vector measures and can use a Radon-Nikodym theorem for vector measures (cf. \cite{Ahmed2013}), to get our desired factorization
    \begin{equation}
        D\hat{h}(X)\mathbf{1}_A=\E\left[\frac{d\mu_{D\hat{h}(X)}\circ X^{-1}}{d\mu}(X)\mathbf{1}_A\right]=:\E\left[\partial_\mu h(\mu)(X) \mathbf{1}_A\right]. \label{eq:Radon Nikodym Representation of Lions derivative}
    \end{equation}
    The relation of vector measures and linear operators, as given in  \eqref{eq:Relation vector measures and Linear Operators}, has long been known and studied (cf. \cite{dinculeanu1967vector}, \cite{diestel1977vector}) and leads to a condition, such that $\partial_\mu h(\mu)(X)\in L^2(\Omega,L(H,V^*))$. This makes the extension from \eqref{eq:Radon Nikodym Representation of Lions derivative} for only indicator functions to general $L^2(\Omega, H)$-random variables possible. Namely, if we define
    \begin{equation*}
        \Lambda_2^{\mathbb{P}}(H):=\left\{S \in L\left(L^2(\Omega, H, \mathbb{P}), V^*\right)\,\middle|\, \|S\|_{2, \mathbb{P}}<\infty\right\},
    \end{equation*}
    where 
    \begin{equation*}
        \|S\|_{2, \mathrm{P}}:=\sup \left\{\sum_{i=1}^n\left\|S\left(\mathbf{1}_{A_i} x_i\right)\right\|_{V^*}\middle| Y=\sum_{i=1}^n \mathbf{1}_{A_i} x_i, A_i \text { disj., } \mathbb{E}\left[\|Y\|_H^2\right] \leq 1\right\},
    \end{equation*}
    the Fréchet derivative being in this space, is the sufficient condition. Thus, we arrive at our definition of Lions differentiability.
    \begin{definition}\label{definition:Lambda-continuously L-differentiable}
        We say that a map $h: \cP_2(H) \rightarrow V^*$ is $\Lambda$-continuously L-differentiable if the lift is continuously Fréchet differentiable, $D \hat{h}(X) \in \Lambda_2^{\mathbb{P}}(H, V^*)$ for all $X \in L^2(\Omega, H)$ and $D \hat{h}: L^2(\Omega, H) \rightarrow \Lambda_2^{\mathbb{P}}(H, V^*)$ is continuous.
    \end{definition}
    We add some measurability properties, that we will need for our coefficients of the adjoint equation to be measurable. 
    \begin{lemma}\label{lemma:joint measurability}
        If $h:V \times \cP_2(H)\to V^*$ is continuously Fréchet differentiable in its first component and $\Lambda$-continuously Fréchet differentiable in its second component, then there exists a jointly measurable version of
        \begin{equation*}
            V\times \cP_2(H)\times H\times H\to V^*,\quad (x,\mu,y,z)\mapsto \partial_\mu h(x,\mu)(y)(z).
        \end{equation*}
    \end{lemma}
    \begin{proof}
        We proceed as  \cite{CarmonaDelarue2018} in Proposition 5.33. Due to the continuous partial Fréchet differentiability, the lift $\hat{h}:V\times L^2(\Omega,H)\to V^*$ is continuously Fréchet differentiable (\cite{Pan2023} Theorem 9.24). Thus, 
        \begin{equation*}
            V\times L^2(\Omega, H)\times L^2(\Omega,H)\to V^*,\quad (x,X,Y)\mapsto \partial_X \hat{h}(x,X)(Y)
        \end{equation*}
        is measurable (even continuous). Further, for every $\epsilon>0$,
        \begin{equation*}
            L^2(\Omega,H)\times H\times H\to L^2(\Omega,H),\quad (X,y,z)\mapsto z\mathbf{1}_{\{\|X-y\|_H\leq \epsilon\}}
        \end{equation*}
        is measurable. Thus, $\psi_\epsilon:V\times L^2(\Omega,H)\times H\times H\to V^*$ defined by
        \begin{equation*}
            \psi_\epsilon(x,X,y,z)= \frac{\partial_X \hat{h}(x,X)(z\mathbf{1}_{\{\|X-y\|_H\leq \epsilon\}})}{\mathbb{P}(\|X-y\|_H \leq \varepsilon)} \mathbf{1}_{\{\mathbb{P}(\|X-y\|_H \leq \varepsilon)>0\}}
        \end{equation*}
        is a measurable mapping. Now, by the definition of the L-derivative and Lebesgue's differentiation theorem, we see, that 
        \begin{align*}
            \psi_\epsilon(x,X,y,z)
            =&\frac{\E\left[\partial_\mu h(x,\cL(X))(X)(z\mathbf{1}_{\{\|X-y\|_H\leq \epsilon\}})\right]}{\mathbb{P}(\|X-y\|_H \leq \varepsilon)} \mathbf{1}_{\{\mathbb{P}(\|X-y\|_H \leq \varepsilon)>0\}}\\
            = &\frac{\int_{\{v\in H:\|v-y\|_H\leq \epsilon\}}\partial_\mu h(x,\cL(X))(v)(z) d\cL(X)(v)}{\cL(X)(\{e\in H:\|e-y\|_H \leq \varepsilon\})} \mathbf{1}_{\{\mathbb{P}(\|X-y\|_H \leq \varepsilon)>0\}}\\
            \to &\partial_\mu h(x,\cL(X))(h)(z).
        \end{align*}
        Using the result from  \cite{BlackwellDubinsSkorokhodExstension1983}, that gives the existence of a jointly measurable 
        \begin{equation*}
            \varphi:[0,1)\times \cP_2(H) \to H,
        \end{equation*}
        such that $\varphi(\cdot,\mu)\sim \mu$, we get a measurable map $\cP_2(H)\to L^2(\Omega,H),\mu\mapsto X^\mu$ such that $\cL(X^\mu)=\mu$, which together with the above gives the result.
    \end{proof}
\subsection{Assumptions}\label{subsection:Assumptions}
    We need a number of assumptions for our results. For the readers convenience, we list them here for later reference. 
    \begin{remark}
        The first three assumptions \ref{assumption:Linear operator}, \ref{Assumption:Standard Assumptions for Existence and Uniqueness of State Equation} and \ref{Assumption:f and g are bounded below and quadratically} are important for the well-posedness of the control problem and will be assumed throughout the whole paper. They give strong existence and uniqueness of a solution to the state equation, an a-priori bound on this solution, Lipschitz continuity of the control-to-state-map and finiteness of the control problem.
        Assumptions \ref{Assumptions:DifferentiabilityOfCoefficients}, \ref{Assumption: f,g diff. and moment bounds} and \ref{Assumption:Convexity of F,B,f} are used for the Pontryagin maximum principle, but will not be needed for the existence result. Assumption \ref{Assumptions:DifferentiabilityOfCoefficients} mainly leads to the Gateaux differentiability of the cost functional and together with Assumption \ref{Assumption: f,g diff. and moment bounds} to the existence and uniqueness of a solution to the adjoint equation, while Assumption \ref{Assumption:Convexity of F,B,f} makes the Hamiltonian convex to allow for the maximum principle.\\
        Later on, Assumption \ref{assumption:V compactly in H} and \ref{assumption:Condition for existence of controls} will be used to show the existence of an optimal (deterministic) control but are not needed for the Pontryagin maximum principle.
    \end{remark}
    \begin{assumption}
    \label{assumption:Linear operator}
        Let $L:H\supset D(H)\to H$ be a linear, densely defined, self-adjoint operator with discrete spectrum, such that there exists $l_1,l_2>0$ such that for all $v\in V$
        \begin{equation*}
            \langle Lv, v\rangle_V \leq l_1\|v\|^2_H-l_2\|v\|^2_V.
        \end{equation*}
    \end{assumption}
    \begin{assumption}
    \label{Assumption:Standard Assumptions for Existence and Uniqueness of State Equation}
        There are constants, $A_1\geq 0$, $\delta>0$, $p\geq 0$ and a
        $c_\cdot \in L^1([0,T],\R)$ such that $q>p+2$, $X_0\in L^q(\Omega,H)$ and the following conditions hold for all $t \in[0, T]$.
        \begin{enumerate}[(H1)]
            \item[(H1)] (Demicontinuity)  For all $\alpha\in U$ and $v \in V$, the map
            \begin{equation*}
                V \times \mathcal{P}_2(H) \ni(u, \mu) \mapsto \langle F(t, u, \mu,\alpha), v\rangle_V
            \end{equation*}
            is continuous. 
            \item[(H2)] (Coercivity) For all $u \in V$, $\mu \in \mathcal{P}_2(H)$ and $\alpha\in U$
            \begin{equation*}
                2\langle Lu+ F(t, u, \mu,\alpha), u\rangle_V
                \leq A_1\left(\|u\|_H^2+\mu\left(\|\cdot\|_H^2\right)+\|\alpha\|_U^2\right)-\delta\|u\|_V^2+c_t.
            \end{equation*}
            \item[(H3a)] (Monotonicity) For all $u, v \in V$, $\mu, \nu \in \mathcal{P}_2(H)$, $\alpha,\beta\in U$ and $t\in[0,T]$
            \begin{align*}
                 &2\langle L(u-v)+F(t, u, \mu,\alpha)-F(t, v, \nu,\beta), u-v\rangle_V\\
                \leq & A_1(\|u-v\|_H^2+\W_2(\mu, \nu)^2 +\|\alpha-\beta\|_U^2)
                -\delta\|u-v\|_V^2.
            \end{align*}
            \item[(H3b)] (Lipschitz-continuity of diffusion) For all $u, v \in V$, $\mu, \nu \in \mathcal{P}_2(H)$, $\alpha,\beta\in U$
            \begin{equation*}
                \|B(t, u, \mu,\alpha)-B(t, v, \nu,\beta)\|_{L_2(H)}^2
                \leq A_1(\|u-v\|_H^2+\W_2(\mu, \nu)^2 +\|\alpha-\beta|_U^2).
            \end{equation*}
            \item[(H4)] (Growth) For all $u \in V, \mu \in \mathcal{P}_2(H)$, $\alpha\in U$ and $t \in[0, T]$,
            \begin{align*}
                \|F(t, u, \mu,\alpha)\|_{V^*}^{2} 
                & \leq (c_t+\|\alpha\|_U^2+\|u\|_V^2+\mu(\|\cdot\|_H^2))(A_1+\|u\|_H^p+\mu(\|\cdot\|_H^2)).
            \end{align*}
        \end{enumerate}
    \end{assumption}
    \begin{assumption}
    \label{Assumption:f and g are bounded below and quadratically}
        \begin{enumerate}[(i)]
            \item $f$ and $g$ are lower bounded and 
            \item there exists $A_2\geq 0$ such that for all $t \in[0, T]$, $x \in H, \alpha \in U$ and $\mu \in \mathcal{P}_q\left(H\right)$
            \begin{align*}
                |f(t, x, \mu, \alpha)| & \leq A_2(1+\|x\|_H^q+\mu(\|\cdot\|_H^q)+\|\alpha\|_U^q) \\
                |g(x, \mu)| & \leq A_2(1+\|x\|_H^q+\mu(\|\cdot\|_H^q)).
            \end{align*}
        \end{enumerate}
    \end{assumption}\noindent
    We already introduced the notation for the $L$-derivative and the lift.
    Further, we write 
    $\hat{F}_x(t,x,\hat{Y},\alpha)=\frac{\partial \hat{F}(t,x,\hat{Y},\alpha)}{\partial x},$ 
    $\hat{F}_\alpha(t,x,\hat{Y},\alpha)=\frac{\partial \hat{F}(t,x,\hat{Y},\alpha)}{\partial \alpha}$, 
    $\hat{F}_Y(t,x,\hat{Y},\alpha)=\frac{\partial \hat{F}(t,x,\hat{Y},\alpha)}{\partial \hat{Y}}$, 
    where $\frac{\partial}{\partial \cdot} $ is a Fréchet derivative in the respective argument, and similarly for $B$, $f$, $g$.
    \begin{assumption}
    \label{Assumptions:DifferentiabilityOfCoefficients} 
        The coefficients $F$ and $B$ are continuously Fréchet differentiable with respect to $(x, \alpha)$ and $\Lambda$-continuously $\mathrm{L}$-differentiable with respect to $\mu \in \mathcal{P}_2\left(H\right)$. Furthermore there exist $\delta^\prime >0$, $p^\prime\geq 0$, $A_3\geq 0$ and a
        $c^\prime_\cdot \in L^1([0,T],\R)$ such that $q>p^\prime+2$ and the following conditions hold for all $t\in[0,T]$.
        \begin{itemize}
            \item[(H2$^\prime$) (Coercivity)] For every $(x, \mu,X, \alpha) \in V \times \mathcal{P}_q\left(H\right) \times L^q\left(\Omega,H\right)\times U$, and every $(z,Z,\beta)\in V\times L^2\left(\Omega,H\right)\times U$
            \begin{align*}
                &2\langle Lz+ F_x(t, x, \mu, \alpha)(z) 
                +\hat{F}_Y(t, x, X, \alpha)(Z), z\rangle_V \\
                \leq& A_3(\|z\|_H^2+\|x\|_H^2+ \|\alpha\|_U^2+\E\left[\|Z\|_H^2+\|X\|_H^2\right]+\mu(\|\cdot\|_H^2))+ c^\prime_t-\delta^\prime \|z\|^{2}_V 
            \end{align*}
            and
            \begin{align*}
                2\langle F_\alpha(t, x, \mu, \alpha)(\beta), z\rangle_V
                \leq A_3(\|z\|_H^2+\|x\|_H^2+ \|\alpha\|_U^2+\|\beta\|_U^2+\mu(\|\cdot\|_H^2))+ c^\prime_t.
            \end{align*}
            \item[(H3$^\prime$)] (Monotonicity) For every $(x, \mu,X, \alpha) \in V \times \mathcal{P}_q\left(H\right) \times L^q\left(\Omega,H\right)\times U$, and every $z_1,z_2\in V$ and $Z_1,Z_2\in  L^2\left(\Omega,H\right)$
            \begin{align*}
                &2\langle (L+ F_x(t, x, \mu, \alpha)) (z_1-z_2),z_1-z_2\rangle_V 
                \leq A_3 \left\| z_1-z_2 \right\|^2_H -\delta^\prime \|z_1-z_2\|_V^2
            \end{align*}
            and
            \begin{align*}
                &2\langle \hat{F}_Y(t,x,X,\alpha)(Z_1-Z_2),z_1-z_2\rangle_V \\
                \leq& A_3 \left(1+\E\left[\|X\|_H^q\right]\right)\left(\E\left[\|Z_1-Z_2\|^2_H\right]\right)+A_3\left\| z_1-z_2 \right\|^2_H -\delta^\prime \|z_1-z_2\|_V^2
            \end{align*}
            \item[(H4$^\prime$)] (Growth) For every $(x, \mu,X, \alpha) \in V \times \mathcal{P}_q\left(H\right) \times L^q\left(\Omega,H\right)\times U$, and every $(z,Z,\beta)\in V\times L^2\left(\Omega,H\right)\times U$:
            \begin{align*}
                \|F_x(t, x, \mu, \alpha)(z)\|^2_{V^*} 
                \leq &  (c^\prime_t+\|\alpha\|_U^2+\|x\|_V^2+\mu(\|\cdot\|_H^2))\\
                \cdot & (A_3+\|x\|_H^{p^\prime}+\mu(\|\cdot\|_H^2))
                +(c^\prime_t+\|z\|^2_V)(A_3+\|z\|^{p^\prime}_H)\\
                \|F_\alpha(t, x, \mu, \alpha)(\beta)\|^2_{V^*} 
                \leq  & (c^\prime_t+\|\beta\|_U^2+\|\alpha\|_U^2+\|x\|_V^2+\mu(\|\cdot\|_H^2))\\
                \cdot&(A_3+\|x\|_H^{p^\prime}+\mu(\|\cdot\|_H^2))\\
                \|\hat{F}_Y(t, x, X, \alpha)(Z)\|^2_{V^*} 
                \leq & (c^\prime_t+\|\alpha\|_U^2+\|x\|_V^2+\E[\|X\|_H^2])\\
                \cdot & (A_3+\|x\|_H^{p^\prime}+\E[\|Z\|_H^2]).
            \end{align*} 
        \end{itemize}
    \end{assumption}    
    \begin{lemma}\label{Lemma:Boundedness of Diffusion derivatives}
        If Assumption \ref{Assumption:Standard Assumptions for Existence and Uniqueness of State Equation} and \ref{Assumptions:DifferentiabilityOfCoefficients} hold, then there is a constant $A_4\geq 0$ such that for every $(x, \mu,X, \alpha) \in V \times \mathcal{P}_q\left(H\right) \times L^q\left(\Omega,H\right)\times U$,
        \begin{align}
            \left\|B_x(t, x, \mu, \alpha)\right\|_{L(H,L_2(H))} & \leq A_4 \label{eq:Boundedness of Diffusion x-derivative}\\
            \left\|B_\alpha(t, x, \mu, \alpha)\right\|_{L(A,L_2(H))} & \leq A_4 \label{eq:Boundedness of Diffusion alpha-derivative}\\
            \left\|\hat{B}_Y(t, x, X, \alpha)\right\|_{L(L^2(\Omega,H),L_2(H))} & \leq A_4
            . \label{eq:Boundedness of Diffusion Y-derivative}
        \end{align}
    \end{lemma}
    \begin{proof}
        This is clear, by the Lipschitz continuity of $B$ from \ref{Assumption:Standard Assumptions for Existence and Uniqueness of State Equation} (H3b) and the Fréchet differentiability of $B$ from Assumption \ref{Assumptions:DifferentiabilityOfCoefficients}.
    \end{proof}
    \begin{assumption}
    \label{Assumption: f,g diff. and moment bounds}
        The coefficients $f$ and $g$ are continuously Fréchet differentiable with respect to $(x, \alpha)$ and $\Lambda$-continuously $\mathrm{L}$-differentiable with respect to $\mu \in \mathcal{P}_2\left(H\right)$ and there exists $A_5\geq 0$ such that for every $(x, \mu,\alpha) \in V \times \mathcal{P}_q\left(H\right) \times U$, $t\in[0,T]$ and $\mu$-almost every $y\in H$
        \begin{align*}
            \|f_x(t, x, \mu, \alpha)\|^2_H & \leq A_5\left(1+\|x\|_H^q+\mu\left(\|\cdot\|_H^q\right)+\|\alpha\|_U^q\right) \\
            \|f_\alpha(t, x, \mu, \alpha)\|^2_U & \leq A_5\left(1+\|x\|_H^q+\mu\left(\|\cdot\|_H^q\right)+\|\alpha\|_U^q\right) \\
            \|\partial_\mu f(t, x, \mu, \alpha)(y)\|^2_H & \leq A_5\left(1+\|x\|_H^q+\mu\left(\|\cdot\|_H^q\right)+\|\alpha\|_U^q+\|y\|^q_H\right) \\
            \|g_x(x, \mu)\|^2_H & \leq A_5\left(1+\|x\|_H^q+\mu\left(\|\cdot\|_H^q\right)\right)\\
            \|\partial_\mu g(x, \mu)(y)\|^2_H & \leq A_5\left(1+\|x\|_H^q+\mu\left(\|\cdot\|_H^q\right)+\|y\|^q_H\right).
        \end{align*}
    \end{assumption}
    \begin{assumption}
    \label{Assumption:Convexity of F,B,f}
        For any $t \in[0, T]$, $x \in V$ and $\mu \in \mathcal{P}_2\left(H\right)$, the functions $U \rightarrow V^* \times L_2(H) \times \mathbb{R}, \alpha \mapsto(F, B,f)(t, x, \mu, \alpha)$ are convex.
    \end{assumption}
    \begin{remark}
        One can restrict the co-domain of the controls further, to some convex (and bounded) $U_1\subset U$. In this case, the assumptions only have to hold for $\alpha,\beta\in U_1$, which might in some cases be the only way to make these work. Further, if one can show, that the state equation only takes values in a certain $H_1\subset H$, then it is sufficient to only make assumptions on this $H_1$, making the assumptions less restrictive. We have omitted these technicalities to make the results more readable, but they can be added without issues.
    \end{remark}
\section{Well-Posedness of the Control Problem}\label{section:Well-Posedness of the Control Problem}
    First of all, we need to show, that a solution to the state equation \eqref{eq:TheBase-State-SPDE} exists and is unique (in the strong sense) for every control $\alpha\in\A$. We define a solution in the following way. 
    \begin{definition} \label{def:Solution to SPDE}
        A continuous $H$-valued $\left(\mathcal{F}_t\right)_{t \geq 0}$-adapted process $\{X_t\}_{t \in[0, T]}$ is a solution of Eq. \eqref{eq:TheBase-State-SPDE}, if for its $\mathbb{P}\otimes dt$-equivalence class $\hat{X}$ and $\bar{X}$ a $V$-valued progressively measurable $d t \otimes \mathbb{P}$-version of $\hat{X}$, it holds
        \begin{equation*}
            \hat{X} \in L^2([0, T] \times \Omega, d t \times \mathbb{P} ; V) \cap L^2([0, T] \times \Omega, d t \times \mathbb{P} ; H),
        \end{equation*}
        and $\mathbb{P}$-a.s. for all $t \in[0, T]$
        \begin{equation*}
            X_t=X_0+\int_0^t L\bar{X}_s+F\left(s, \bar{X}_s, \cL(\bar{X}_s),\alpha_s\right) d s+\int_0^t B\left(s, \bar{X}_s, \cL(\bar{X}_s),\alpha_s\right) d W_s.
        \end{equation*}
    \end{definition}\noindent
    From now on, we will write $X^\alpha$ for a solution to \eqref{eq:TheBase-State-SPDE} to emphasize the dependence on the control $\alpha$. The following theorem gives the existence and uniqueness of a solution.
    \begin{theorem}[\cite{HongHuLiu2022} Theorem 3.1] \label{theorem:Existence of solution (Hong2022)}
        Suppose that Assumption \ref{assumption:Linear operator} and \ref{Assumption:Standard Assumptions for Existence and Uniqueness of State Equation} hold. Then for every control $\alpha\in\A$, \eqref{eq:TheBase-State-SPDE} has a unique solution in the sense of Definition \ref{def:Solution to SPDE} and there exists $C>0$ such that
        \begin{equation*}
            \mathbb{E}\left[\sup _{s \in[0, t]}\|X^\alpha_s\|_H^2+\int_0^t\|X^\alpha_s\|_V^2 d s\right]
            \leq C\left(1+\mathbb{E}\left\|X_0\right\|_H^2\right).
        \end{equation*}
    \end{theorem}\noindent
    We get a stronger a priori estimate by using the higher integrability of the controls.
    \begin{lemma}\label{Lemma:a priori-Estimate}
        If Assumption \ref{assumption:Linear operator} and \ref{Assumption:Standard Assumptions for Existence and Uniqueness of State Equation} hold, there is a constant $C$ such that
        \begin{equation*}
            \mathbb{E}\left[\sup _{t \in[0, T]}\left\|X^\alpha_t\right\|_{H}^q+\left(\int_0^T\left\|X^\alpha_t\right\|_{V}^2 \mathrm{~d} t\right)^{q/2}\right] \leq C \E\left[1+\int_0^T\|\alpha(t)\|_U^q \mathrm{~d} t\right]
            \leq C(1+K).
        \end{equation*}
    \end{lemma}
    \begin{proof}
        We will denote $\theta^\alpha_s=(s, X_s^\alpha, \cL(X^\alpha_s),\alpha_s)$. We have, using Itô's formula (\cite{Liu2015} Theorem 4.2.5) and (H2),
        \begin{align}
            \left\|X^\alpha_t\right\|_{H}^2
            =&\left\|X_0\right\|_H^2+\int_0^t2\langle L X_s^\alpha +F\left(\theta^\alpha_s\right), X_s^\alpha\rangle_V+\|B\left(\theta^\alpha_s\right)\|_{L_2(H)}^2 d s\nonumber\\
            &+2 \int_0^t\langle X_s^\alpha, B\left(\theta^\alpha_s\right) d W_s\rangle_H\nonumber\\
            \leq& \left\|X_0\right\|_H^2+\int_0^tA_1(\|X_s^\alpha\|_H^2+ \cL(X_s^\alpha)\left(\|\cdot\|_H^2\right)+\|\alpha_s\|_U^2)-\delta\|X_s^\alpha\|_V^2+c_s d s \nonumber\\
            & +2 \int_0^t\langle X_s^\alpha, B\left(\theta^\alpha_s\right) d W_s\rangle_H \label{eq:StartingItoInequality}.
        \end{align}
        Leaving out the $\delta\|X^\alpha_s\|_V^2$-term, taking both sides to the power $q/2$, taking the supremum with respect to $t \in[0, T]$ and taking expectations yields
        \begin{align*}
            \mathbb{E}\left[\sup _{t \in[0, T]}\left\|X^\alpha_t\right\|_{H}^q\right]
            &\leq  c_1\E\left[\left\|X_0\right\|_H^q+\int_0^T \sup_{s\in[0,t]} \|X_s^\alpha\|_H^q d s + \int_0^T\|\alpha_s\|_U^q ds\right] \\
            & +c_1\left(\int_0^T|c_s| d s\right)^{q/2}  
            +c_1\E\left[ \sup _{t \in[0, T]}
            2 \left| \int_0^t\langle X_s^\alpha, B\left(\theta^\alpha_s\right) d W_s\rangle_H\right|^{q/2}\right].
        \end{align*}
        Using the Burkholder-Davis-Gundy inequality, (H3b), the Young-inequality and $q\geq 2$, we get
        \begin{align*}
            \E\left[ \sup _{t \in[0, T]}
            \left| \int_0^t\langle X_s^\alpha, B\left(\theta^\alpha_s\right) d W_s\rangle_H\right|^{\frac{q}{2}}\right]
            \leq & c_2 \mathbb{E}\left[1+\int_0^T \sup_{s\in[0,T]}\|X_s^\alpha\|_H^q+\|\alpha_s\|_U^q ds\right].
        \end{align*}
        Putting both inequalities together and using Fubini results in 
        \begin{align*}
            \mathbb{E}\left[\sup _{t \in[0, T]}\left\|X^\alpha_t\right\|_{H}^q\right]
            \leq & c_3\E\left[\left\|X_0\right\|_H^q
            + \int_0^T \sup_{s\in[0,t]} \|X_s^\alpha\|_H^q +\|\alpha_s\|_U^q ds +\left(\int_0^t|c_s| ds \right)^{q/2}\right]\\
            \leq & c_4
            +c_3\E\left[ \int_0^T \|\alpha_s\|_U^qd s
            +\int_0^T\sup_{s\in[0,t]} \|X_s^\alpha\|_H^q d s\right]
        \end{align*}
        and, since $c_4+c_3\E[\int_0^t\|\alpha_s\|_U^qd s]$ is monotone increasing in $t$, using Gronwall gives 
        \begin{equation*}
            \mathbb{E}\left[\sup _{t \in[0, T]}\left\|X^\alpha_t\right\|_{H}^q\right] \leq c_5 e^{Tc_3}\left(1 +\int_0^T\E\left[\|\alpha_s\|_U^q\right]d s\right) .
        \end{equation*}
        Pulling $-\delta\|X_s^\alpha\|_V^2$ in \eqref{eq:StartingItoInequality} to the other side, leaving out $\|X_s^\alpha\|_H^2$ on the left hand side and repeating the arguments yields the same estimation for $\E[(\int_0^T\|X^\alpha_t\|_{V}^2 \mathrm{~d} t)^{q/2}]$.
    \end{proof}\noindent
    We will now show, that the control-to-state-map is Lipschitz continuous. 
    \begin{lemma} \label{lemma:Lipschitz of Control-to-State-Map}
        If Assumption \ref{assumption:Linear operator} and \ref{Assumption:Standard Assumptions for Existence and Uniqueness of State Equation} hold, then the control-to-state-map
        \begin{equation*}
            G: L^2(\Omega\times [0, T] ,U) \supset \A\rightarrow L^2\left(\Omega,C([0,T],H)\right)\cap L^2\left(\Omega,L^2([0,T],V)\right),\quad \alpha \mapsto X^\alpha,
        \end{equation*}
        is well-defined and Lipschitz continuous from $L^2(\Omega\times [0, T],U)$ to $L^2(\Omega,C([0,T],H))$ and to $L^2(\Omega,L^2([0,T],V))$, i.e. there exists a constant $C$
        \begin{equation*}
            \E\left[\sup_{t\in[0,T]}\left\|X^\alpha_t-X^\beta_t\right\|_H^2
            +\int_0^T \left\|X^\alpha_t-X^\beta_t\right\|_V^2dt
            \right]
            \leq C \E\left[\int_0^T\left\|\alpha_t-\beta_t\right\|_U^2 \mathrm{~d} s\right].
        \end{equation*}
    \end{lemma}
    \begin{proof}
        We again use $\theta^\alpha_s=(s, X_s^\alpha, \cL(X_s^\alpha),\alpha_s)$. Firstly, since we have a unique strong solution for every $\alpha\in\A$ by Theorem \ref{theorem:Existence of solution (Hong2022)}, the map is well-defined. Towards the second claim, we get, using Itô's formula (\cite{Liu2015} Theorem 4.2.5) and (H3a-b),
        \begin{align}
            \left\|X^\alpha_t-X^\beta_t\right\|_{H}^2
            =&\int_0^t2_{V^*}\langle L(X^\alpha_s-LX^\beta_s)+ F\left(\theta^\alpha_s\right)-F(\theta^\beta_s), X_s^\alpha-X_s^\beta\rangle_V \nonumber\\
            +&\|B(\theta^\alpha_s)-B(\theta^\beta_s)\|_{L_2(H)}^2 d s+2 \int_0^t\langle X_s^\alpha-X_s^\beta, (B(\theta^\alpha_s)-B(\theta^\beta_s)) d W_s\rangle_H \nonumber\\
            \leq&  \int_0^t A_1\left\|X_s^\alpha-X_s^\beta\right\|_H^2+A_1 \W_2\left(\cL(X_s^\alpha), \cL(X_s^\beta)\right)^2 +A_1\left\|\alpha_s-\beta_s\right\|_U^2   \label{eq:StartingItoInequalityLipschitzContinuity1}\\
            -& \delta\|X^\alpha_s-X^\beta_s\|^2_V ds +2 \int_0^t\langle X_s^\alpha-X_s^\beta, (B(\theta^\alpha_s)-B(\theta^\beta_s)) d W_s\rangle_H \label{eq:StartingItoInequalityLipschitzContinuity2}.
        \end{align}
        Leaving out $- \delta\|X^\alpha_s-X^\beta_s\|^2_V$ in \eqref{eq:StartingItoInequalityLipschitzContinuity1}-\eqref{eq:StartingItoInequalityLipschitzContinuity2}, taking the supremum with respect to $t \in[0, T]$, and then taking expectations yields
        \begin{align*}
            \mathbb{E}\left[\sup _{t \in[0, T]}\left\|X^\alpha_t-X^\beta_t\right\|_{H}^2\right]
            \leq & c_1\E\left[\int_0^T \sup_{s\in[0,t]} \|X_s^\alpha-X_s^\beta\|_H^2 d t +\int_0^T\|\alpha_s-\beta_s\|_U^2 d s\right]\\
            &+ c_1\E\left[ \sup_{t\in[0,T]}\left|\int_0^t\langle X_s^\alpha-X_s^\beta, (B(\theta^\alpha_s)-B(\theta^\beta_s)) d W_s\rangle_H\right|\right].
        \end{align*}
        Using the Burkholder-Davis-Gundy inequality, (H3b) and the Young inequality, we get
        \begin{align*}
            &\E\left[ \sup_{t\in[0,T]}\left|\int_0^t\langle X_s^\alpha-X_s^\beta, (B(\theta^\alpha_s)-B(\theta^\beta_s)) d W_s\rangle_H\right|\right]\\
            \leq & \frac{1}{2}\mathbb{E}\left[\sup_{s \in[0,T]}\left\|X_s^\alpha-X_s^\beta\right\|^2_{H}\right]
            +c_2\E\left[\int_0^T \sup_{t \in[0,s]}\left\|X^\alpha_t-X^\beta_t\right\|_H^2+\|\alpha_s-\beta_s\|_U^2 \mathrm{~d} s\right]
        \end{align*}
        Now, the Gronwall inequality gives an estimate on $\E[\sup_{t\in[0,T]}\|X^\alpha_t-X^\alpha_t\|_H^2]$. Pulling $- \delta\|X^\alpha_s-X^\beta_s\|^2_V$ in \eqref{eq:StartingItoInequalityLipschitzContinuity2} to the other side gives the result.
    \end{proof}\noindent
    Lastly, we notice, that the optimization problem is finite.
    \begin{lemma} \label{lemma:well-posedness of J}
        If Assumption \ref{assumption:Linear operator}, \ref{Assumption:Standard Assumptions for Existence and Uniqueness of State Equation} and \ref{Assumption:f and g are bounded below and quadratically} hold, then $J:\A\to \R$ is well-defined.
    \end{lemma}
    \begin{proof}
        First, by Assumption \ref{Assumption:f and g are bounded below and quadratically} (i) $J$ is lower bounded. Secondly, by Assumption \ref{Assumption:f and g are bounded below and quadratically} (ii) the a priori-estimate in Lemma \ref{Lemma:a priori-Estimate} on $X^\alpha_t$ and the boundedness of $\mathbb{A}$, $J$ is upper bounded.
    \end{proof}

\section{Regularity of Control-to-State-Map and Cost Functional}\label{section:Regularity of Control-to-State-Map and Cost Functional}
    As we have already seen in our short introduction to the Lions derivative in section \ref{subsection:Lions derivative}, the lift of a function $h:\cP_2(H)\to H$ has a very important role. Since we assumed our probability space to be atomless, we can always lift to this very same probability space without trouble. For every $\mu\in \cP_2(H)$ we will find an $X\in L^2(\Omega,H)$, such that $\mu=\cL(X)$. We will write
    \begin{equation*}
        \hat{F}:[0,T]\times H \times L^2(\Omega,H)\times U\to V^*, (t,h,\hat{Y},\alpha)\mapsto F(t,h,\cL(\hat{Y}),\alpha)
    \end{equation*}
    for the lift of $F$ and $B,f,g$ will be treated in the same way. Also, we will adopt the notation, that if we have some random variables $X,Y,Z$ and write $\hat{X},\hat{Y}$ and $\hat{Z}$, we mean copies on the same probability space, such that $\cL(X,Y,Z)=\cL(\hat{X},\hat{Y},\hat{Z})$. Mostly, we will be using this, to make clear, that we have lifted a function, by using a copy in the argument, i.e. if $X\sim \mu\in\cP_2(H)$
    \begin{equation*}
        h(\mu)=\hat{h}(X)=:\hat{h}(\hat{X}).
    \end{equation*}
    This makes it easier to distinguish between functions, that take the whole random variable as an argument (distribution dependence) and functions, that take the random variable evaluated at some point $\omega\in\Omega$, e.g. for $X\in L^2(\Omega,H)$ with $\cL(X)=\mu$
    \begin{equation*}
        (F(t,X,\mu,\alpha))(\omega)
        =\hat{F}(t,X(\omega),\hat{X},\alpha(\omega))
        :=\hat{F}(t,X(\omega),X,\alpha(\omega)).
    \end{equation*}
    We will also adapt the notation $\theta^\alpha_t(\omega)=(t, X^\alpha_t(\omega), \cL(X^\alpha_t), \alpha_t(\omega))$ to the lifting by writing $\hat{\theta}^\alpha_t(\omega)=(t,X^\alpha_t(\omega),\hat{X}^\alpha_t,\alpha_t(\omega))$.
    \subsection{Gateaux Derivative of the Control-to-State-Map}\label{section:Gateaux Derivative of the Control-to-State-Map}
        \begin{theorem}
            If Assumptions \ref{assumption:Linear operator}, \ref{Assumption:Standard Assumptions for Existence and Uniqueness of State Equation}, \ref{Assumption:f and g are bounded below and quadratically} and \ref{Assumptions:DifferentiabilityOfCoefficients} hold, the control-to-state map
            \begin{equation*}
                G: L^2(\Omega\times [0, T] ,U) \supset \A\rightarrow L^2\left(\Omega,C([0,T],H)\right),\quad \alpha \mapsto X^\alpha,
            \end{equation*}
            is Gâteaux-differentiable and its derivative at $\alpha \in \mathbb{A}$ in direction $\beta \in \mathbb{A}$ is given by
            \begin{equation*}
                d G(\alpha) \cdot \beta=Z^{\alpha, \beta},
            \end{equation*}
            where the process $Z=Z^{\alpha, \beta}$ is characterized as the unique solution (in the sense of Definition \ref{def:Solution to SPDE}) to
            \begin{equation}
                \begin{aligned}
                d Z_t= & LZ_t+ F_x(\theta^\alpha_t) (Z_t)+F_\alpha(\theta^\alpha_t)  (\beta_t)+\hat{F}_Y(\hat{\theta}^\alpha_t)(\hat{Z}_t) d t \\
                & +B_x(\theta^\alpha_t) (Z_t)+B_\alpha(\theta^\alpha_t) (\beta_t)+\hat{B}_Y(\hat{\theta}^\alpha_t)(\hat{Z}_t) d W_t  \\
                Z_0= & 0
            \end{aligned} \label{eq:Gateaux derivative SPDE}
            \end{equation}
        \end{theorem}
        \begin{remark}
            \eqref{eq:Gateaux derivative SPDE} is only a true McKean-Vlasov equation if coupled with \eqref{eq:TheBase-State-SPDE}, as $\hat{F}_Y(\hat{\theta}^\alpha_t)(\hat{Z}_t)$ depends not only on $\cL(Z)$, but on $\cL(X,Z)$ (and the same holds for $\hat{B}_Y$). Taking the same notion of a solution for this equation as in Definition \ref{def:Solution to SPDE}, we will show existence and uniqueness in the following proof.\\
            Note, that we could use the Lions derivative to write the equation, as
            \begin{align*}
                (\hat{F}_Y(\hat{\theta}^\alpha_t)(\hat{Z}_t))(\omega)
                =&\hat{\E}\left[\partial_\mu F(t,X^\alpha_t(\omega),\cL(X^\alpha_t),\alpha_t(\omega))(\hat{X}^\alpha_t)(\hat{Z}_t)\right]\\
                :=&\int_{H^2} \partial_\mu F(t,X^\alpha_t(\omega),\cL(X^\alpha_t),\alpha_t(\omega))(x)(z) \cL(X^\alpha_t,Z_t)(d(x,z)),
            \end{align*}
            but intentionally do not and instead write the equation using the Fréchet derivative of the lift, to emphasize, that we do not yet need the Lions derivative.
        \end{remark}
        \begin{proof}
            \textbf{Existence and uniqueness of a solution to the equation:} For this, we first fix $\hat{Z}$, see that the resulting equations have a unique solution for each $\hat{Z}$ and then make a fixed point argument to get the desired solution.
            We want to define
            \begin{equation*}
                \Psi: C\left([r, T], L^2\left(\Omega,H\right)\right)\rightarrow C\left([r, T] , L^2\left(\Omega,H\right)\right), \tilde{Z}\to Z^{\tilde{Z}},
            \end{equation*}
            where $Z^{\tilde{Z}}$ is the unique solution to \eqref{eq:Gateaux derivative SPDE} (started in $r\in [0,T]$) for this fixed $\hat{Z}=\tilde{Z}$ and $Z_r\in L^q(\Omega, H)$. Thus, we must check the existence and uniqueness, but letting
            \begin{align*}
                & \tilde{A}(t, \omega, v)
                :=Lv+ F_x(\theta^\alpha_t(\omega)) (v)+F_\alpha(\theta^\alpha_t(\omega))  (\beta_t)
                +\hat{F}_Y(\hat{\theta}^\alpha_t(\omega))(\tilde{Z}_t) \\
                &\tilde{B}(t, \omega, v)
                :=B_x(\theta^\alpha_t(\omega)) (v)+B_\alpha(\theta^\alpha_t(\omega)) (\beta_t)+\hat{B}_Y(\hat{\theta}^\alpha_t(\omega))(\tilde{Z}_t)
            \end{align*}
            we get $dZ^{\tilde{Z}}_t = \tilde{A}(t,Z^{\tilde{Z}}_t) dt +\tilde{B}(t,Z^{\tilde{Z}}_t) dW_t$. Using Assumption \ref{Assumptions:DifferentiabilityOfCoefficients}, it follows then by classical SPDE results (cf. \cite{Liu2015}) that the equation has a unique solution. Note, that $(t,\omega)\to \hat{F}_Y(\hat{\theta}^\alpha_t(\omega))(\tilde{Z}_t) $ is progressively measurable, as it is a limit of progressively measurable functions, and the same holds for $\hat{B}_Y$ (also cf. Lemma \ref{lemma:joint measurability}).\\
            For the fixed point argument, we define the metric 
            \begin{equation*}
                \delta(Z^{(1)}, Z^{(1)}):=\sup _{t \in[r, T]} e^{-\kappa t} \E\left[\|Z^{(1)}_t- Z^{(2)}_t\|_H^2\right].
            \end{equation*}
            For $\tilde{Z}^{(1)},\tilde{Z}^{(2)}\in C([r, T], L^2(\Omega,H))$, Itô's formula yields
            \begin{align*}
                \delta(Z^{\tilde{Z}^{(1)}},Z^{\tilde{Z}^{(2)}})
                = &\sup_{t\in[r,T]}\int_r^t e^{-\kappa s}
                \E\left[ 2\langle \tilde{A}(t, Z^{\tilde{Z}^{(1)}}_s) 
                - \tilde{A}(t, Z^{\tilde{Z}^{(2)}}_s), Z^{\tilde{Z}^{(1)}}_s-Z^{\tilde{Z}^{(2)}}_s \rangle_V\right. \\
                &\left.+\left\|\tilde{B}(t, Z^{\tilde{Z}^{(1)}}_s) 
                - \tilde{B}(t, Z^{\tilde{Z}^{(2)}}_s)\right\|_{L_2(H)}^2
                -\kappa \left\|Z^{\tilde{Z}^{(1)}}_s-Z^{\tilde{Z}^{(2)}}_s\right\|^2_H\right] ds.
            \end{align*}
            Now, Assumption \ref{Assumptions:DifferentiabilityOfCoefficients} (H3$^\prime$), Lemma \ref{Lemma:Boundedness of Diffusion derivatives} and choosing $\kappa$ appropriately gives
            \begin{align*}
                &\delta\left(Z^{\tilde{Z}^{(1)}},Z^{\tilde{Z}^{(2)}}\right)
                \leq  \left(A_3\int_r^T 1+\E\left[\left\|X^\alpha_s\right\|_H^q\right]ds+A_4\right)\sup_{t\in[r,T]}e^{-\kappa t}\E\left[\|\tilde{Z}^{(1)}_t-\tilde{Z}^{(2)}_t\|_H^2 \right].
            \end{align*}
            By the a priori estimate from Lemma \ref{Lemma:a priori-Estimate} $(A_3\int_r^T 1+\E[\|X^\alpha_s\|_H^q]ds+A_4)<1$ for $T$ sufficiently close to $r$, which makes $\Psi$ a contraction. Therefore a unique solution on $[r,\tau]$ for $\tau$ small enough exists and the needed 'smallness' does not depend on the starting time $r$ and therefore this solution can be extended to the whole interval $[0,T]$.
            Notice, that due to the conditions from Assumption \ref{Assumptions:DifferentiabilityOfCoefficients}, we can get the same estimate as in Lemma \ref{Lemma:a priori-Estimate} for $Z$ by replicating the arguments, i.e.
            \begin{equation}
                \mathbb{E}\left[\sup _{t \in[0, T]}\left\|Z_t^\alpha\right\|_{H}^q+\left(\int_0^T\left\|Z_t^\alpha\right\|_{V}^2 \mathrm{~d} t\right)^{q/2}\right] \leq c_2\label{eq:estimate for Z}.
            \end{equation}
            \textbf{Gateaux differentiability of the control-to-state map:} Let now $\alpha, \beta \in \mathbb{A}$ and $\epsilon>0$ such that $\alpha+\epsilon \beta \in \mathbb{A}$. By $X$ we denote the solution of the state equation \eqref{eq:TheBase-State-SPDE} with respect to $\alpha$ and by $X^\epsilon$ we denote the solution to \eqref{eq:TheBase-State-SPDE} with respect to $\alpha+\epsilon \beta$. Furthermore for $\lambda \in[0,1]$ we introduce the linear interpolations $X^{\lambda, \epsilon}:=X+\lambda\left(X^\epsilon-X\right)$ and $\alpha^{\lambda, \epsilon}:=\alpha+\lambda \epsilon \beta$. We will use $\theta_t = (t,X_t,\cL(X_t),\alpha_t)$ and $\hat{\theta}_t$ for the lifted variables and $\theta^\epsilon$ and $\theta^{\lambda,\epsilon}$ for the corresponding tuples $(t,X^\epsilon,\cL(X_t^\epsilon),\alpha+\epsilon\beta)$ and $(t,X^{\lambda,\epsilon},\cL(X^{\lambda,\epsilon}),\alpha^{\lambda,\epsilon})$ respectively.\\
            $\hat{F}(t,\cdot,\cdot,\cdot)$ is continuously Fréchet differentiable for every $t\in[0,T]$ and therefore 
            \begin{align*}
                \hat{F}(\hat{\theta}_t^\epsilon)
                =& \hat{F}(\hat{\theta}_t)
                +\int_0^1 \hat{F}_x(\hat{\theta}^{\lambda,\epsilon}_t)(X^\epsilon_t-X_t) 
                +\hat{F}_Y(\hat{\theta}^{\lambda,\epsilon}_t)(\hat{X}^\epsilon_t-\hat{X}_t)
                +\hat{F}_\alpha(\hat{\theta}^{\lambda,\epsilon}_t)(\epsilon\beta_t) d \lambda.
            \end{align*}
            and similarly for $B$.
            Defining the difference
            \begin{equation*}
                \Delta_t^\epsilon:=\frac{X_t^\epsilon-X_t}{\epsilon}-Z_t^{\alpha, \beta},
            \end{equation*}
            we can plug in \eqref{eq:TheBase-State-SPDE}, the above expansion and \eqref{eq:Gateaux derivative SPDE}, to get
            \begin{align*}
                d\Delta^\epsilon_t
                =&\left( L\Delta^\epsilon_t + \int_0^1 F_x(\theta^{\lambda,\epsilon}_t)\Delta^\epsilon_t+\hat{F}_Y(\hat{\theta}^{\lambda,\epsilon}_t)\hat{\Delta}^\epsilon_td\lambda\right)
                + \left(\int_0^1 F_x(\theta^{\lambda,\epsilon}_t)d \lambda
                -F_x(\theta_t)\right)Z_t\\
                &+\left(\int_0^1 F_\alpha(\theta^{\lambda,\epsilon}_t) d \lambda
                -F_\alpha(\theta_t)\right)\beta_t
                +\left(\int_0^1 \hat{F}_Y(\hat{\theta}^{\lambda,\epsilon}_t) d\lambda
                -\hat{F}_Y(\hat{\theta}_t)\right)\hat{Z}_t dt\\
                &+ \int_0^1 B_x(\theta^{\lambda,\epsilon}_t)\Delta^\epsilon_t+ \hat{B}_Y(\hat{\theta}^{\lambda,\epsilon}_t)\hat{\Delta}^\epsilon_td\lambda
                + \left(\int_0^1 B_x(\theta^{\lambda,\epsilon}_t)d \lambda
                -B_x(\theta_t)\right)Z_t\\
                &+\left(\int_0^1 B_\alpha(\theta^{\lambda,\epsilon}_t) d \lambda
                -B_\alpha(\theta_t)\right)\beta_t
                +\left(\int_0^1 \hat{B}_Y(\hat{\theta}^{\lambda,\epsilon}_t) d\lambda
                -\hat{B}_Y(\hat{\theta}_t)\right)\hat{Z}_t dW_t\\
                =:& I_1(t)+I_2(t)+I_3(t)+I_4(t)dt + I_5(t)+I_6(t)+I_7(t)+I_8(t)dW_t.
            \end{align*}
            Now, using Itô's formula, taking the supremum over $t\in[0,T]$ and the expectation and then using the Burkholder-Davis-Gundy inequality and Young inequality, we arrange the terms in a suitable way to estimate them. We arrive at
            \begin{align}
                &\E\left[\sup_{t\in[0,T]}\|\Delta^\epsilon_t\|_H^2\right] \nonumber\\
                \leq & \E\left[\sup_{t\in[0,T]} \int_0^t 2\langle I_1(t),\Delta_t^\epsilon \rangle_V +\left\|I_5(t)\right\|^2_{L_2(H)}dt 
                +\delta^\prime\int_0^T\left\|\Delta_t^\epsilon \right\|_V^2 dt\right] \label{eq:DeltaEpsilon Estimation 1}\\
                &+c_1\E\left[\int_0^T \left\| I_2(t)+I_3(t)+I_4(t)\right\|_{V^*}^2 +\left\|I_6(t)+I_7(t)+I_8(t)\right\|^2_{L_2(H)}dt\right] \label{eq:DeltaEpsilon Estimation 2}\\
                &+ c_1\E\left[\left(\int_0^T \|\Delta^\epsilon_t\|_H^2 \left\|I_5(t)+I_6(t)+I_7(t)+I_8(t)\right\|^2_{L_2(H)} dt\right)^{\frac{1}{2}}\right] \label{eq:DeltaEpsilon Estimation 3}.
            \end{align}
            With Assumption \ref{Assumptions:DifferentiabilityOfCoefficients} (H3$^\prime$), Lemma \ref{Lemma:Boundedness of Diffusion derivatives} and the a priori bound, we get 
            \begin{equation*}
                \eqref{eq:DeltaEpsilon Estimation 1}
                \leq \int_0^T c_2\E\left[\sup_{s\in[0,t]}\|\Delta^\epsilon_s\|_H^2\right] dt.
            \end{equation*}
            Further, 
            \begin{align*}
                \eqref{eq:DeltaEpsilon Estimation 3} \leq &  \frac{1}{2} \E\left[\sup_{t\in[0,T]}\|\Delta^\epsilon_t\|_H^2\right]+c_3\E\left[\int_0^T \sum_{i=5}^8 \left\| I_i(t)\right\|^2_{L_2(H)} dt\right],
            \end{align*}
            and by Lemma \ref{Lemma:Boundedness of Diffusion derivatives} \eqref{eq:Boundedness of Diffusion x-derivative} and \eqref{eq:Boundedness of Diffusion Y-derivative} and the a priori bound
            \begin{align*}
                \E\left[\int_0^T \left\|I_5(t)\right\|^2_{L_2(H)} dt\right] 
                \leq &c_4 \int_0^T \E\left[\sup_{s\in[0,t]}\|\Delta^\epsilon_s\|_H^2\right]dt.
            \end{align*}
            Putting this together, we get
            \begin{align*}
                \frac{1}{2}\E\left[\sup_{t\in[0,T]}\|\Delta^\epsilon_t\|_H^2\right] 
                \leq &  \int_0^T c_5\E\left[\sup_{s\in[0,t]}\|\Delta^\epsilon_s\|_H^2+\sum_{i=2}^4 \left\| I_i(t)\right\|_{V^*}^2 +\sum_{i=6}^8 \left\| I_i(t)\right\|^2_{L_2(H)}\right] dt
            \end{align*}
            and using Gronwall results in 
            \begin{equation*}
                \E\left[\sup_{t\in[0,T]}\|\Delta^\epsilon_t\|_H^2\right] 
                \leq  c_6\E\left[\int_0^T \sum_{i=2}^4 \left\| I_i(t)\right\|_{V^*}^2 +\sum_{i=6}^8 \left\| I_i(t)\right\|^2_{L_2(H)}dt\right].
            \end{equation*}
            It is left, to show, that the right-hand side converges to $0$ for $\epsilon\to 0$. We will only consider the convergence of $I_2$ as the other terms can be done in a similar way. Firstly, 
            \begin{align*}
                \E\left[\int_0^T \|I_2\|_{V^*}^2dt\right]
                \leq & \int_0^1\E\left[\int_0^T\left\|\left( F_x(\theta^{\lambda,\epsilon}_t)
                -F_x\left(\theta_t\right)\right)Z_t\right\|_{V^*}^2 dt\right]d \lambda.
            \end{align*}
            Now, by (H4$^\prime$), Young's inequality, $p^\prime+2 \leq q$ and the a priori bounds on $X$ and $Z$, we see
            \begin{align*}
                &\E\left[\int_0^T\left\|\left( F_x(\theta^{\lambda,\epsilon}_t)
                -F_x\left(\theta_t\right)\right)Z_t\right\|_{V^*}^2 dt\right]\\ 
                \leq & \sup_{\lambda \in [0,1]}2\E\left[\sup_{t\in[0,T]}\left(A_3+\|X^{\lambda,\epsilon}_t\|_H^{p^\prime}+\E\left[\|X^{\lambda,\epsilon}_t\|_H^2\right]\right)^{\frac{q}{q-2}} \right.\\
                & \left.+\left(\int_0^T \left(c^\prime_t+\|\alpha^{\lambda,\epsilon}_t\|_U^2+\|X^{\lambda,\epsilon}_t\|_V^2+\E\left[\|X^{\lambda,\epsilon}_t\|_H^2\right]\right)dt\right)^{\frac{q}{2}}\right]\\
                &+2 \E\left[\sup_{t\in[0,T]} \left(A_3+\|Z_t\|^{p^\prime}_H\right)^\frac{q}{q-2}+\left(\int_0^T \left(c^\prime_t+\|Z_t\|^2_V\right)dt\right)^\frac{q}{2}\right]<\infty,
            \end{align*}
            so we only have to show, that for every $\lambda $ and $\epsilon\to 0$
            \begin{equation*}
                \E\left[\int_0^T\left\|\left( F_x(\theta^{\lambda,\epsilon}_t)
                -F_x\left(\theta_t\right)\right)Z_t\right\|_{V^*}^2 dt\right]
                \to 0.
            \end{equation*}
            But the above estimation also works for a slightly bigger exponent, as $q>p^\prime+2$, so for some $\gamma>0$
            \begin{equation*}
                \int_0^T \E\left[\left\|\left( F_x(\theta^{\lambda,\epsilon}_t)
                -F_x(\theta_t)\right)Z_t\right\|_{V^*}^{2(1+\gamma)} \right] dt<\infty,
            \end{equation*}
            and we get uniform integrability of 
            the integrand. Furthermore, by Lemma \ref{lemma:Lipschitz of Control-to-State-Map},
            we get $\theta^{\lambda,\epsilon}_t\to \theta_t$ in $[0,T]\times L^2(\Omega, H)\times \cP_2(H)\times U$ for all $t\in[0,T]$ and also $X^{\lambda,\epsilon}\to X$ in $L^2(\Omega,L^2([0,T],V))$. Thus, we have convergence in measure (with respect to $\P\otimes dt$) of $\theta^{\lambda,\epsilon}_t(\omega)\to \theta_t(\omega)$ in $[0,T]\times  V\times \cP_2(H)\times U$ and by the continuity of $F_x(t,\cdot,\cdot,\cdot): V\times \cP_2(H)\times U$, as $F$ is continuously differentiable, the desired convergence.
        \end{proof}
    \subsection{Gateaux Differentiability of the Cost Functional}
        An application of the chain rule gives the Gateaux differentiability of the cost functional.
        \begin{corollary} \label{Corollary:Representation of Gateaux of Cost Function}
            Let Assumptions \ref{assumption:Linear operator}, \ref{Assumption:Standard Assumptions for Existence and Uniqueness of State Equation}, \ref{Assumption:f and g are bounded below and quadratically} and \ref{Assumptions:DifferentiabilityOfCoefficients} hold. Then, the cost functional $J: \mathbb{A} \rightarrow \mathbb{R}$ is Gâteaux differentiable and its Gâteaux derivative at $\alpha \in \mathbb{A}$ in direction $\beta \in \mathbb{A}$ is given by
            \begin{align*}
                d J(\alpha) \cdot \beta= & \mathbb{E}\left[\int_0^T f_x\left(t, X^\alpha_t, \cL\left(X^\alpha_t\right), \alpha_t\right)\left( Z_t^{\alpha, \beta}\right)
                +f_\alpha\left(t, X^\alpha_t, \cL\left(X^\alpha_t\right), \alpha_t\right) \left(\beta_t\right)dt\right] \\
                & +\mathbb{E}\left[\int_0^T\hat{\mathbb{E}}\left[\partial_\mu f \left(t, X^\alpha_t, \cL(X^\alpha_t), \alpha_t\right)\left(\hat{X}^\alpha_t\right) \left(\hat{Z}_t^{\alpha, \beta}\right)\right]dt\right]\\
                & +\mathbb{E}\left[ g_x\left(X^\alpha_T, \cL\left(X^\alpha_T\right)\right) \left( Z_T^{\alpha, \beta}\right)+\hat{\mathbb{E}}\left[\partial_\mu g\left(X^\alpha_T, \cL\left(X^\alpha_t\right)\right)\left(\hat{X}^\alpha_T\right) \left( \hat{Z}_T^{\alpha, \beta}\right)\right]\right].
            \end{align*}
        \end{corollary}
\section{Adjoint Calculus and the Pontryagin Maximum Principle}\label{section:Adjoint Calculus and the Pontryagin Maximum Principle}
    \subsection{The Adjoint Equation and Its Well-Posedness}\label{subsection:The Adjoint Equation and Its Well-Posedness}
        We need to introduce the (first-order) adjoint equation. For this, we again remember our notation $\theta^\alpha_t=(t, X^\alpha_t, \cL(X^\alpha_t), \alpha_t)$, where $X^\alpha$ is the unique solution to \eqref{eq:TheBase-State-SPDE} for the control $\alpha\in\A$, and change the notation $\hat{\theta}^\alpha_t=(t,X^\alpha_t,\hat{X}^\alpha_t,\alpha_t)$, which was suited for a lifting of a function to an $L^2$-space, to $\hat{\theta}^\alpha_t=(t, \hat{X}^\alpha_t, \cL(X^\alpha_t), \hat{\alpha}_t)$, which is suited to work with L-derivatives and just denotes an independent copy of $\theta^\alpha_t$. This change might seem confusing at first, but it should always be clear from context whether the function, where we plug it in, is lifted or not. This will mostly be used in the following way:
        \begin{align*}
            \hat{\E}\left[F_{\mu}(\hat{\theta}^\alpha_t)(X^\alpha_t)(\hat{Z}_t)\right](\omega)
            =&\hat{\E}\left[\partial_\mu F(t,\hat{X}^\alpha_t,\cL(X^\alpha_t),\hat{\alpha}_t)(X^\alpha_t(\omega))(\hat{Z}_t)\right]\\
            :=&\int_{H^2} \partial_\mu F(t,x,\cL(X^\alpha_t),\alpha)(X^\alpha_t(\omega))(z) \pi(d(x,z,\alpha)),
        \end{align*}
        where $\pi=\cL(X^\alpha_t,Z_t,\alpha_t)$, and similarly for $B,f$ and $g$, i.e. the copies are integrated out. As $\alpha$ is fixed in this section, we will also just use $\theta_t$, if no confusion is possible.\\
        We denote an adjoint operator with a $*$. The adjoint equation is given by
        \begin{align}
            d P_t= & -L P_t- F_x^*\left(\theta^\alpha_t\right)P_t
            -B_x^*\left(\theta^\alpha_t\right)Q_t
            -f_x\left(\theta^\alpha_t\right)-\hat{\mathbb{E}}\left[ \left(\partial_\mu F(\hat{\theta}^\alpha_t)(X^\alpha_t)\right)^*\hat{P}_t\right] \label{eq:Adjoint Equation 1}\\
            & -\E\left[\left(\partial_\mu B(\hat{\theta}^\alpha_t)(X^\alpha_t)\right)^*\hat{Q}_t\right]
            -\hat{\mathbb{E}}\left[  \partial_\mu f(\hat{\theta}^\alpha_t)( X^\alpha_t)\right] d t+Q_t d W_t \label{eq:Adjoint Equation 2}\\
            P_T= & g_x\left(X^\alpha_T, \cL\left(X^\alpha_T\right)\right)+\hat{\mathbb{E}}\left[\partial_\mu g\left(\hat{X}_T, \cL\left(X^\alpha_T\right)\right)(X^\alpha_T)\right].\label{eq:Adjoint Equation 3}
        \end{align}
        We introduce the Hamiltonian, for simpler notation,
        \begin{equation*}
            H(t, x, \mu,\alpha, p, q)
            :=\langle Lx + F(t, x, \mu, \alpha), p\rangle_V+\langle B(t, x, \mu, \alpha), q\rangle_{L_2(H)}+f(t, x, \mu, \alpha)
        \end{equation*}
        and can rewrite
        \begin{equation}
            \begin{aligned}
                d P_t & =-H_x\left(\theta^\alpha_t, P_t, Q_t\right)
                -\hat{\mathbb{E}} \left[\partial_\mu H(\hat{\theta}^\alpha_t, \hat{P}_t, \hat{Q}_t)\left(X^\alpha_t\right)\right] d t+Q_t d W_t\\
                P_T & =g_x\left(X^\alpha_T, \cL\left(X^\alpha_T\right)\right)
                +\hat{\mathbb{E}} \left[\partial_\mu g(\hat{X}_T, \cL\left(X^\alpha_T\right))\left(X^\alpha_T\right)\right].
            \end{aligned}
            \label{eq:adjoint with Hamiltonian formulation 1}
        \end{equation}
        Note, that we suppress the dependence of $(P,Q)$ on $\alpha$. A more precise notation would be $(P^\alpha,Q^\alpha)$. This is a backwards SPDE of McKean-Vlasov type (technically again only if we couple it with \eqref{eq:TheBase-State-SPDE} since the coefficients depend on the joint law $\cL(X,P,Q)$), so we lack a definition of a solution to such an equation.
        \begin{definition} \label{def:solution to adjoint}
            A tuple of a continuous $H$-valued adapted processes $P$ and an $L_2(H)$-valued adapted process $Q$ is called a solution to \eqref{eq:adjoint with Hamiltonian formulation 1}, if for its 
            $\P\otimes dt$-equivalence class $(\tilde{P},\tilde{Q})$ we have 
            \begin{align*}
                \tilde{P}\in &L^2(\Omega \times [0, T] , H)\cap L^2(\Omega \times [0, T] , V)\text{ and }
                \tilde{Q}\in L^2(\Omega\times [0, T] , L_2(H))
            \end{align*}
            and $\P$-a.s. in $V^*$
            \begin{equation*}
                P_t = P_T+\int_t^T H_x\left(\theta^\alpha_s, \bar{P}_s, \bar{Q}_s\right)
                +\hat{\mathbb{E}} \left[\partial_\mu H(\hat{\theta}^\alpha_s, \hat{\bar{P}}_s, \hat{\bar{Q}}_s)\left(X_s\right)\right] ds 
                - \int_t^T \bar{Q}_s dW_s,
            \end{equation*}
            where $(\bar{P},\bar{Q})$ is a $V\times L_2(H)$-valued progressively measurable $\P\otimes dt$-version of $(\tilde{P},\tilde{Q})$.
        \end{definition}
        \begin{remark}
            Note, that we cannot use the lift of $F$ and $B$ for the formulation of the adjoint equation, as the input in the L-derivative is not integrated out here. This approach really uses the factorization of the L-derivative and pulls that function out of its expectation.\\
            Further, note that the coefficients of \eqref{eq:adjoint with Hamiltonian formulation 1} are indeed measurable in a suitable way by Lemma \ref{lemma:joint measurability}, which is a priori not clear.
        \end{remark}
        \begin{theorem}\label{theorem:adjoint existence and uniqueness}
            Let Assumptions \ref{assumption:Linear operator}, \ref{Assumption:Standard Assumptions for Existence and Uniqueness of State Equation}, \ref{Assumption:f and g are bounded below and quadratically}, \ref{Assumptions:DifferentiabilityOfCoefficients} and \ref{Assumption: f,g diff. and moment bounds} hold. Then, there exists a unique solution $(P, Q)$ to \eqref{eq:adjoint with Hamiltonian formulation 1} in the sense of Definition \ref{def:solution to adjoint}.
        \end{theorem}
        \begin{proof}
            First, we fix $(\hat{P},\hat{Q})\in L^2(\Omega\times[0,T],H)\times L^2(\Omega\times[0,T],L_2(H))$ and show the existence and uniqueness of the resulting SPDE, that has no distribution dependence anymore. Afterwards, we will again use a fixed point argument to get our desired solution. For this purpose, we define 
            \begin{align*}
                \phi_t
                = &-\hat{\mathbb{E}}\left[ \langle \partial_\mu F(\hat{\theta}^\alpha_t)(X^{\alpha}_t)(\cdot),\hat{P}_t\rangle_V+\langle \partial_\mu B(\hat{\theta}^\alpha_t)(X^{\alpha}_t)(\cdot),\hat{Q}_t\rangle_{L_2(H)}
                +\partial_\mu f(\hat{\theta}^\alpha_t)( X^{\alpha}_t)\right].
            \end{align*}
            \textbf{Galerkin approximation:} We start, by projecting onto an orthonormal basis (of $H$) of eigenvectors of $L$
            \begin{equation*}
                V_n=\operatorname{span}\left\{e_1,\dots,e_n\right\},
            \end{equation*}
            i.e. $\Pi_n:V^*\to V_n: y\mapsto \sum_{i=1}^n \langle y, e_i\rangle_V e_i$, such that $\operatorname{span}\left\{e_i|i\in\N\right\}$ is dense in $V$. $\Pi_n|_H$ is the orthogonal projection onto $V_n\subset H$. We further define
            \begin{equation*}
                W^{n}_t:=\sum_{i=1}^n\langle W_t, e_i\rangle_H e_i\quad \text{and}\quad \cF^n_t = \sigma\left(W^n_s,s\leq t\right) \vee \cN,
            \end{equation*}
            and 
            \begin{align*}
                \tilde{F}^n_t(\omega)
                :=&\E[\Pi_n F^*_x(\theta_t)\mid \cF^n_t](\omega),\quad
                \tilde{B}^n_t(\omega)
                :=\E[\Pi_n B^*_x(\theta_t)\mid \cF^n_t](\omega),\\
                \tilde{f}^n_t(\omega)
                :=&\E[\Pi_n f_x(\theta_t)\mid \cF^n_t](\omega), \quad
                \tilde{\phi}^n_t(\omega)
                :=\E[\Pi_n \phi_t\mid \cF^n_t](\omega)\quad\text{and}\\
                \tilde{P}^n_T 
                = &\E\left[\Pi_n\left( g_x\left(X^\alpha_T, \cL\left(X^\alpha_T\right)\right)
                +\tilde{\mathbb{E}} \left[\partial_\mu g\left(\tilde{X}_T, \cL\left(X^\alpha_T\right)\right)\left(X^\alpha_T\right)\right]\right)\middle|\cF^n_T\right].
            \end{align*}
            Now, consider
            \begin{align*}
                dP^n_t =& -\left(LP^n_t+\tilde{F}^n_tP^n_t 
                + \tilde{B}^n_t Q^n_t
                + \tilde{f}^n_t
                +\tilde{\phi}^n_t \right)dt 
                + Q^n_t dW^n_t\\
                P^n_T=&\tilde{P}^n_T,
            \end{align*}
            which is a BSDE on $V_n$ and identifying this finite dimensional space with $\R^n$, we can use standard results for finite dimensional BSDE with monotone coefficients. Using our Assumptions \ref{Assumptions:DifferentiabilityOfCoefficients} and \ref{Assumption: f,g diff. and moment bounds} and the bounds on $X$ and the control space, we can show that this equation fulfills (BSDE-MH$_\cF$) from \cite{Pardoux2014} and therefore, using Theorem 5.27 from there, we get the existence of a unique solution $(P^n,Q^n)\in L^2(\Omega, C([0,T],V_n))\times L^2(\Omega,L^2([0,T],L_2(H)))$, that is progressively measurable with respect to $\{\cF^n_t\}_{t\geq 0}$, for each $n\in\N$.\\
            \textbf{A priori bound:} We want an energy estimate, to extract a (weakly) convergent subsequence and determine the limit as our desired solution. For this, we notice, that by Itô's formula (\cite{Liu2015} Theorem 4.2.5)
            \begin{align*}
                \|P^n_T\|_H^2
                =&\|P^n_t\|_H^2-\int_t^T 2\langle LP^n_s+\tilde{F}^n_sP^n_s 
                + \tilde{B}^n_s Q^n_s
                + \tilde{f}^n_s
                +\tilde{\phi}^n_s ,P^n_s\rangle_H ds\\
                &+ \int_t^T\|Q_s^n\|^2_{L_2(H)} ds 
                + \int_t^T 2\langle P^n_s,Q^n_s dW^n_s\rangle_H.
            \end{align*}
            Rearranging, taking expectations and using Assumption \ref{Assumptions:DifferentiabilityOfCoefficients} (H2$^\prime$) gives
            \begin{align}
                &\E\left[\|P^n_t\|_H^2 +\int_t^T\|Q_s^n\|^2_{L_2(H)} +\delta^\prime\|P^n_s\|_V^2 ds\right]\nonumber\\
                =&\E\left[\|P^n_T\|_H^2+\int_t^T 2\langle LP^n_s+\tilde{F}^n_sP^n_s 
                + \tilde{B}^n_s Q^n_s
                + \tilde{f}^n_s
                +\tilde{\phi}^n_s ,P^n_s\rangle_H
                +\delta^\prime\|P^n_s\|_V^2
                ds\right]\nonumber\\
                \leq & \E\left[\|P^n_T\|_H^2+\int_t^T C_s+c_1\|P^n_s\|_H^2+\frac{1}{2}\|Q^n_s\|_{L_2(H)}^2ds\right] \label{eq:Adjoint_Galerkin_a priori_ExpectationEstimation}
            \end{align}
            where $(C_t)\in L^1([0,T]\times\Omega,\R^+)$ exists, depending on $X,\cL(X),\hat{Q},\hat{P},\alpha$ and $c^\prime$, but not on $n$. In particular 
            \begin{equation*}
                \E\left[\|P^n_t\|_H^2\right]
                \leq \E\left[ \|P_T\|_H^2+\int_0^T C_s ds \right] +c_1 \int_t^T\E\left[ \|P^n_s\|_H^2\right]ds,
            \end{equation*}
            and, using a backwards Gronwall inequality (cf. \cite{FAN2011427} Lemma 4), we deduce
            \begin{align*}
                &\E\left[\|P^n_t\|_H^2\right]
                \leq \E\left[ \|P_T\|_H^2+\int_0^T C_s ds \right] e^{\int_t^T c_1 ds} <\infty,
            \end{align*}
            where $P_T$ fulfills the necessary integrability by the moment estimates on the derivative of $g$ from Assumption \ref{Assumption: f,g diff. and moment bounds} and the a priori estimates on $X$. This implies the uniform bound
            \begin{equation*}
                \sup_{n\in\N}\E\left[\int_0^T\|P^n_s\|_H^2 ds\right] \leq C(T) \E\left[ \|P_T\|_H^2+\int_0^T A_3 C_s ds \right]<\infty
            \end{equation*}
            and by \eqref{eq:Adjoint_Galerkin_a priori_ExpectationEstimation}, also uniform bounds on $\E[\int_0^T\|Q_s^n\|^2_{L_2(H)}ds] $ and $\E[\int_0^T\|P^n_s\|_V^2 ds]$.\\
            \textbf{Limiting Procedure:} With these estimates, we can find subsequences of $(P^n)$ and $(Q^n)$, that we will denote with the same index, such that
            \begin{itemize}
                \item $P^n \to P$ weakly in $L^2(\Omega\times [0,T],V)$ and weakly in $L^2(\Omega\times [0,T],H)$ and
                \item $Q^n\to Q$ weakly in $L^2(\Omega\times [0,T],L_2(H))$.
            \end{itemize}
            These weak limits are again progressively measurable, since the approximants are.\\
            Notice, that $h_t\mapsto \int_t^T h_s dW_s$ is a bounded linear operator from the square integrable, progressively measurable, $L_2(H)$-valued processes into the square integrable continuous martingales on $H$ by the Itô-isometry. Thus, this mapping is also weakly-weakly continuous and therefore 
            \begin{equation*}
                \int_t^T Q^n_s dW^n_s = \int_t^T Q^n_s \Pi_n dW_s \to \int_t^T Q_s dW_s
            \end{equation*}
            weakly on the square integrable continuous martingales and hence also weakly$^*$ in $L^\infty ([0,T],L^2(\Omega,H))$.\\
            By these convergences we get for all $k,l\in\N$, $v \in V_k \subset V\subset H$ and $\psi\in L^\infty ([0,T]\times\Omega,\R)$ progressively measurable with respect to $\{\cF_t^l\}_{t\geq 0}$, using Fubini, that
            \begin{align*}
                &\E\left[\int_0^T \langle P_t,\psi_t v\rangle_V dt\right]\\
                =& \lim_{n\to\infty} \E\left[\int_0^T \langle P^n_T,\psi_t v\rangle_V dt \right.\\
                &+\left.\int_0^T \int_t^T \langle LP^n_s+\tilde{F}^n_sP^n_s 
                + \tilde{B}^n_s Q^n_s
                + \tilde{f}^n_s
                +\tilde{\phi}^n_s,\psi_t v\rangle_V ds 
                -\langle \int_t^T Q_s^n dW^n_s,\psi_t v\rangle_H dt\right]\\
                =& \lim_{n\to\infty} \E\left[\langle P^n_T,v\rangle_V \int_0^T \psi_t dt -\int_0^T\langle \int_t^T Q_s^n dW^n_s,\psi_t v\rangle_H dt
                 \right.\\
                & \left.+\int_0^T \langle LP^n_s+\tilde{F}^n_sP^n_s 
                + \tilde{B}^n_s Q^n_s
                + \tilde{f}^n_s
                +\tilde{\phi}^n_s,\int_0^s \psi_t dt v \rangle_V ds\right],\\
                =& \lim_{n\to\infty} \E\left[\langle P^n_T,v\rangle_V \int_0^T \psi_t dt -\int_0^T\langle \int_t^T Q_s^n dW^n_s,\psi_t v\rangle_H dt\right.\\
                &  +\int_0^T \langle \int_0^s \psi_t dt\left(Lv+F_x(\theta_s)v\right) ,P^n_s\rangle_V 
                + \langle \int_0^s \psi_t dt B_x(\theta_s)v, Q^n_s\rangle_{L_2(H)}+ \langle \Pi_n f_x(\theta_s)\\
                &+\left.\Pi_n \phi(s),v \int_0^s \psi_t dt\rangle_H ds 
                \right],\\
                =&\E\left[\int_0^T \langle P_T,\psi_t v\rangle_V dt-\int_0^T\langle \int_t^T Q_s dW_s,\psi_t v\rangle_H dt \right.\\
                &+\left.\int_0^T \int_0^s \langle LP_s+F^*_x(\theta_s)P_s 
                + B^*_x(\theta_s) Q_s
                + f_x(\theta_s)
                +\phi_s,\psi_t v\rangle_V dt ds 
                \right],
            \end{align*}
            hence
            \begin{align*}
                P_t = P_T+\int_t^T LP_s+F^*_x(\theta_s)P_s 
                + B^*_x(\theta_s) Q_s
                + f_x(\theta_s)
                +\phi_s ds 
                - \int_t^T Q_s dW_s
            \end{align*}
            $\P\otimes dt$-almost surely. Thus, defining 
            \begin{align*}
                \tilde{P}_t = P_T+\int_t^T LP_s+F^*_x(\theta_s)P_s 
                + B^*_x(\theta_s) Q_s
                + f_x(\theta_s)
                +\phi_s ds 
                - \int_t^T Q_s dW_s,
            \end{align*}
            we see that $(\tilde{P},Q)$ is a solution to the equation.\\
            Towards the uniqueness, let $(P^{(1)},Q^{(1)})$ and $(P^{(2)},Q^{(2)})$ be two solutions. Writing $\Delta^P_t:=P^{(1)}_t-P^{(2)}_t$ and $\Delta^Q_t:=Q^{(1)}_t-Q^{(2)}_t$ and using Itô's formula, we get
            \begin{align*}
                0=\E\left[\left\|\Delta^P_T\right\|_H^2 \right] 
                = &\E\left[\left\|\Delta^P_t\right\|_H^2\right]
                +\E\left[\int_t^T 2\langle -L\left(\Delta^P_s\right)-F^*_x(\theta_s)\left(\Delta^P_s\right) 
                \right. \\
                &\left.-B_x^*(\theta_s)\left(\Delta^Q_s\right)
                ,\Delta^P_s\rangle_V +\left\|\Delta^Q_s\right\|^2_{L_2(H)} ds\right],
            \end{align*}
            rearranging and using (H3$^\prime$), Lemma \ref{Lemma:Boundedness of Diffusion derivatives} and the Young inequality, results in 
            \begin{align*}
                \E\left[\|\Delta^P_t\|_H^2 +\int_t^T\|\Delta^Q_s\|^2_{L_2(H)} ds\right]
                =&\int_t^T \E\left[ c_2\|\Delta^P_s\|_H^2+\frac{1}{2}\| \Delta^Q_s\|_{L_2(H)}^2 \right] ds.
            \end{align*}
            Now using the backwards Gronwall inequality again, we get the uniqueness. The continuity of $P$ follows from the Itô formula from \cite{Liu2015} Theorem 4.2.5 after noticing 
            \begin{align*}
                P_t = P_0+\int_0^t -LP_s-F^*_x(\theta_s)P_s 
                - B^*_x(\theta_s) Q_s
                - f_x(\theta_s)
                -\phi_s ds 
                + \int_0^t Q_s dW_s,
            \end{align*}
            where 
            \begin{align*}
                P_0 = P_T+\int_0^T LP_s+F^*_x(\theta_s)P_s 
                + B^*_x(\theta_s) Q_s
                + f_x(\theta_s)
                +\phi_s ds 
                - \int_0^T Q_s dW_s.
            \end{align*}
            \textbf{Fixed point argument:} Now, that we have a unique solution for fixed $(\hat{P},\hat{Q})$, we denote $\left(P^{(\hat{P},\hat{Q})},Q^{(\hat{P},\hat{Q})}\right)$
            this solution and define $S:=C\left([t, T], L^2\left(\Omega,H\right)\right)\times L^2([t,T]\times\Omega,L_2(H))$ and
            \begin{equation*}
                \Psi: S\rightarrow S, (\hat{P},\hat{Q})\to \left(P^{(\hat{P},\hat{Q})},Q^{(\hat{P},\hat{Q})}\right).
            \end{equation*}
            We will show, that this is a contraction for the metric 
            \begin{equation*}
                \delta((\hat{P}^1,\hat{Q}^1), (\hat{P}^2,\hat{Q}^2)):=\sup _{s \in[t, T]} e^{\kappa s} \E\left[\left\|\hat{P}^1_s- \hat{P}^2_s\right\|_H^2\right]+\E\left[\int_t^Te^{\kappa s} \left\| \hat{Q}^1_s-\hat{Q}^2_s\right\|_{L_2(H)}^2\right],
            \end{equation*}
            where $\kappa>0$ will be chosen appropriately later on.
            Denoting $\Delta^P_t:=P^{(\hat{P}^1,\hat{Q}^1)}_t-P^{(\hat{P}^2,\hat{Q}^2)}_t$ and $\Delta^Q_t:=Q^{(\hat{P}^1,\hat{Q}^1)}_t-Q^{(\hat{P}^2,\hat{Q}^2)}_t$ and using Itô's formula, we see that
            \begin{align*}
                0=&e^{\kappa T}\E\left[\left\|\Delta^P_T\right\|_H^2 \right] \\
                = &e^{\kappa t}\E\left[\left\|\Delta^P_t\right\|_H^2\right]
                +\E\left[\int_t^T e^{\kappa s} 2\langle -L\left(\Delta^P_s\right)-F^*_x(\theta_s)\left(\Delta^P_s\right) 
                -B_x^*(\theta_s)\left(\Delta^Q_s\right)\right. \\
                &- \left.\hat{\E}\left[\partial_\mu F^*(\hat{\theta}_s)(X^\alpha_s)\left(\hat{P}^1_s-\hat{P}^2_s\right) 
                + \partial_\mu B^*(\hat{\theta}_s)(X^\alpha_s)\left(\hat{Q}^1_s-\hat{Q}^2_s\right)\right]
                ,\Delta^P_s\rangle_V\right.\\
                &\left.+ e^{\kappa s}\left\|\Delta^Q_s\right\|^2_{L_2(H)} ds\right]
                +\kappa \int_t^T e^{\kappa s}\E\left[\left\|\Delta^P_s\right\|_H^2 \right] ds
            \end{align*}
            and rearranging and using Fubini gives
            \begin{align*}
                &\E\left[e^{\kappa t}\left\|\Delta^P_t\right\|_H^2 +\int_t^T e^{\kappa s}\left\|\Delta^Q_s\right\|^2_{L_2(H)} ds\right]\\
                =&\int_t^T e^{\kappa s}\E\left[ 2\langle L\left(\Delta^P_s\right)+F_x(\theta_s)\left(\Delta^P_s\right),\Delta^P_s\rangle_V 2\langle B_x(\theta_s)\left(\Delta^P_s\right), \Delta^Q_s\rangle_{L_2(H)} \right]\\
                &+ 2 e^{\kappa s}\hat{\E}\left[ \E\left[\langle \partial_\mu F(\hat{\theta}_s)(X^\alpha_s)\left(\Delta^P_s\right),\hat{P}^1_s-\hat{P}^2_s\rangle_V\right.\right.\\
                &\left.\left.+\langle \partial_\mu B(\hat{\theta}_s)(X^\alpha_s)\left(\left(\Delta^P_s\right)\right)
                ,\hat{Q}^1_s-\hat{Q}^2_s\rangle_{L_2(H)}\right]\right]
                - \kappa e^{\kappa s}\E\left[\left\|\Delta^P_s\right\|_H^2 \right] ds.
            \end{align*}
            Now, using (H3$^\prime$), (H4$^\prime$), Lemma \ref{Lemma:Boundedness of Diffusion derivatives}, Youngs's inequality and the a priori bound, we arrive at
            \begin{align*}
                &\E\left[e^{\kappa t}\left\|\Delta^P_t\right\|_H^2 +\int_t^T e^{\kappa s}\left\|\Delta^Q_s\right\|^2_{L_2(H)} ds\right]\\
                \leq & \int_t^T e^{\kappa s}\E\left[ c_3\left\| \Delta^P_s \right\|^2_H +\frac{1}{2}\left\| \Delta^Q_s\right\|_{L_2(H)}^2 \right]- \kappa e^{\kappa s}\E\left[\left\|\Delta^P_s\right\|_H^2 \right] ds\\
                &+ e^{\kappa s}\hat{\E}\left[ 
                c_3\left\|\hat{P}^1_s-\hat{P}^2_s\right\|^2_H+\frac{1}{4}\left\| \hat{Q}^1_s-\hat{Q}^2_s\right\|_{L_2(H)}^2 \right],
            \end{align*}
            and, choosing $\kappa$ appropriately, this leads to 
            \begin{align*}
                &\E\left[e^{\kappa t}\|\Delta^P_t\|_H^2 +\frac{1}{2}\int_t^T e^{\kappa s}\|\Delta^Q_s\|^2_{L_2(H)} ds\right]\\
                \leq &
                \int_t^T c_3 ds
                \sup_{s\in[t,T]}e^{\kappa s} \hat{\E}\left[\|\hat{P}^1_s-\hat{P}^2_s\|^2_H\right]
                +\frac{1}{4}\hat{\E}\left[\int_t^T e^{\kappa s}\| \hat{Q}^1_s-\hat{Q}^2_s\|_{L_2(H)}^2 ds\right].
            \end{align*}
            We see that, $\Psi$ is a contraction if $t$ is close enough to $T$. But we can extent this solution (backwards) to a full solution by doing multiple steps.
        \end{proof}
    \subsection{Representation of the Gradient of the Cost Functional}\label{subsection:Representation of the Gradient of the Cost Functional}
        \begin{corollary} \label{Corollary:Ito Product Formula for Z and P}
            Let Assumptions \ref{assumption:Linear operator}, \ref{Assumption:Standard Assumptions for Existence and Uniqueness of State Equation}, \ref{Assumption:f and g are bounded below and quadratically}, \ref{Assumptions:DifferentiabilityOfCoefficients} and \ref{Assumption: f,g diff. and moment bounds} hold. Then, for $\alpha,\beta\in\A$ and $X^\alpha$ the solution to the state equation \eqref{eq:TheBase-State-SPDE}, $Z$ the solution to equation \eqref{eq:Gateaux derivative SPDE} and $(P,Q)$ the solution to the adjoint equation \eqref{eq:Adjoint Equation 1}-\eqref{eq:Adjoint Equation 3}, it holds
            \begin{align*}
                \mathbb{E}\left[\langle P_T, Z_T\rangle_H\right]
                =\mathbb{E}&\left[\int_0^T\langle  F_\alpha\left(\theta_t\right)(\beta_t), P_t\rangle_V+\langle B_\alpha\left(\theta_t\right) (\beta_t), Q_t\rangle_{ L_2(H)} \right.\\
                &\left. - f_x\left(\theta_t\right) \left(Z_t\right)
                - \hat{\mathbb{E}} \left[\partial_\mu f\left(\theta_t\right) \left(\hat{X}^\alpha_t\right)\left(\hat{Z}_t\right)\right]dt\right].
            \end{align*}
        \end{corollary}
        \begin{proof}
            An application of the Itô product rule, taking the expectation and using Fubini results in
            \begin{align*}
                &\mathbb{E}\left[\langle P_T, Z_T\rangle_H\right] \\
                =& \E\left[\int_0^T \langle F_\alpha(\theta^\alpha_t)  (\beta_t)+\hat{\E}\left[\partial_\mu F(\theta^\alpha_t)(\hat{X}^\alpha_t)(\hat{Z}_t)\right] , P_t\rangle_V d t-f_x\left(t, X^\alpha_t, \cL\left(X^\alpha_t\right), \alpha_t\right) \right.\\
                & -\hat{\mathbb{E}}\left[ \langle \partial_\mu F(\hat{\theta}^\alpha_t)(X^\alpha_t)(Z_t),\hat{P}_t\rangle_V
                +\langle \partial_\mu B(\hat{\theta}^\alpha_t)(X^\alpha_t)(Z_t),\hat{Q}_t\rangle_{L_2(H)}\right]\\
                &\left.-\hat{\mathbb{E}}\left[  \partial_\mu f(\hat{\theta}^\alpha_t)( X^\alpha_t)\right] d t + \int_0^T \langle B_\alpha(\theta^\alpha_t) (\beta_t)+\hat{\E}\left[\partial_\mu B(\theta^\alpha_t)(\hat{X}^\alpha_t)(\hat{Z}_t)\right] ,  Q_t\rangle_{L_2(H)} dt\right]\\
                =&\mathbb{E}\left[\int_0^T\langle  F_\alpha\left(\theta_t\right)(\beta_t), P_t\rangle_V+\langle B_\alpha\left(\theta_t\right) (\beta_t), Q_t\rangle_{ L_2(H)} \right.\\
                &\left. - f_x\left(\theta_t\right) \left(Z_t\right)
                - \hat{\mathbb{E}} \left[\partial_\mu f\left(\theta_t\right) (\hat{X}^\alpha_t)(\hat{Z}_t)\right]dt\right].
            \end{align*}
        \end{proof}
        \begin{corollary}
            Let Assumptions \ref{assumption:Linear operator}, \ref{Assumption:Standard Assumptions for Existence and Uniqueness of State Equation}, \ref{Assumption:f and g are bounded below and quadratically}, \ref{Assumptions:DifferentiabilityOfCoefficients} and \ref{Assumption: f,g diff. and moment bounds} hold. Then, for $\alpha,\beta\in\A$ and $X^\alpha$ the solution to the state equation \eqref{eq:TheBase-State-SPDE} and $(P,Q)$ the solution to the adjoint equation \eqref{eq:Adjoint Equation 1}-\eqref{eq:Adjoint Equation 3}, it holds
            \begin{equation*}
                d J(\alpha) \cdot \beta
                =\mathbb{E}\left[\int_0^T H_\alpha\left(t, X^\alpha_t, \cL\left(X^\alpha_t\right),\alpha_t, P_t, Q_t\right) \left(\beta_t\right) d t\right].
            \end{equation*}
        \end{corollary}
        \begin{proof}
            Combining the above Corollary \ref{Corollary:Ito Product Formula for Z and P} and the terminal condition
            \begin{equation*}
                P_T= g_x\left(X^\alpha_T, \cL\left(X^\alpha_T\right)\right)+\hat{\mathbb{E}}\left[\partial_\mu g\left(X^\alpha_T, \cL\left(X^\alpha_T\right)\right)(\hat{X}^\alpha_T)\right]
            \end{equation*}
            we get
            \begin{align*}
                &\mathbb{E}\left[\langle g_x\left(X^\alpha_T, \cL\left(X^\alpha_T\right)\right)+\hat{\mathbb{E}}\left[\partial_\mu g\left(X^\alpha_T, \cL\left(X^\alpha_T\right)\right)(\hat{X}^\alpha_T)\right], Z_T\rangle_H\right]\\
                =&\mathbb{E}\left[\int_0^T\langle  F_\alpha\left(\theta^\alpha_t\right)(\beta_t), P_t\rangle_V+\langle B_\alpha\left(\theta^\alpha_t\right) (\beta_t), Q_t\rangle_{ L_2(H)} - f_x\left(\theta^\alpha_t\right) \left(Z_t\right)\right.\\
                &\left.- \hat{\mathbb{E}} \left[\partial_\mu f\left(\theta^\alpha_t\right) \left(\hat{X}^\alpha_t\right)\left(\hat{Z}_t\right)\right]dt\right].
            \end{align*}
            Plugging this into the representation of the Gateaux derivative of $J$ from Corollary \ref{Corollary:Representation of Gateaux of Cost Function}, we immediately get the result.
        \end{proof}
    \subsection{The Pontryagin Maximum Principle}\label{subsection:The Pontryagin Maximum Principle}
        \begin{theorem}\label{theorem:Pontryagin Maximum Principle}
            Let Assumptions \ref{assumption:Linear operator}, \ref{Assumption:Standard Assumptions for Existence and Uniqueness of State Equation}, \ref{Assumption:f and g are bounded below and quadratically}, \ref{Assumptions:DifferentiabilityOfCoefficients}, \ref{Assumption: f,g diff. and moment bounds} and \ref{Assumption:Convexity of F,B,f} hold. Further, let $\alpha^* \in \mathbb{A}$ be an optimal control for the problem. If $X^{\alpha^*}$ is the solution to the state equation \eqref{eq:TheBase-State-SPDE} and $(P, Q)$ is the solution to the corresponding adjoint equation, then we have for all $\alpha \in U$
            \begin{equation*}
                H\left(t, X^{\alpha^*}_t, \cL(X^{\alpha^*}_t), \alpha^*_t, P_t, Q_t\right) \leq H\left(t, X^{\alpha^*}_t, \cL(X^{\alpha^*}_t), \alpha, P_t, Q_t\right)\quad \P\otimes dt-\text{a.s.}
            \end{equation*}
        \end{theorem}
        \begin{proof}
            By the optimality of $\alpha^*$ we get $d J(\alpha^*) \cdot(\beta-\alpha^*)\geq 0$ for any $\beta\in\A$ and hence by the convexity from Assumption \ref{Assumption:Convexity of F,B,f}
            \begin{equation*}
                \int_0^T \mathbb{E}\left[H\left(t, X^{\alpha^*}_t, \cL(X^{\alpha^*}_t), \beta_t, P_t, Q_t\right)-H\left(t, X^{\alpha^*}_t, \cL(X^{\alpha^*}_t), \alpha^*_t, P_t, Q_t\right)\right] d t \geq 0.
            \end{equation*}
            Let $C \subset[0, T] \times \Omega$ be a progressively measurable set
            and $\alpha \in A$.  Plugging
            \begin{equation*}
                \beta_t= \begin{cases}\alpha & \text { for } t \in C \\ \alpha^*_t & \text { otherwise. }\end{cases}
            \end{equation*}
            which is now an admissible control, into the above inequality, gives the result.
        \end{proof}

\section{Existence of a Deterministic Optimal Control} \label{section:existence of an optimal control}
    We want to prove, that, at least, when we only allow for deterministic controls, there is at least one optimal control, that minimizes the cost functional. We use a compactness method, that is used by \cite{Pardoux2021} to prove existence of solutions to SPDE. For this, we will show tightness of the laws of an approximating sequence on the path spaces $\Omega_1=L^2([0,T],V)_{\text{weak}}$, $\Omega_2=C([0,T],H_{\text{weak}})$ and $\Omega_3=L^2([0,T],H)$ and then extract a weakly convergent subsequence.
    According to  \cite{Pardoux2021}, this method is justified in a thesis of M. Viot, that is not published, where the justification for the application of a Prokhorov theorem is given. We even want a sequence of almost surely convergent random variables with the same laws, so we apply the generalized Skorokhod embedding theorem from \cite{Jakubowski1998} (cf. Appendix \refforappendixone). Our control space is now
    \begin{equation*}
        \alpha\in \mathbb{A}:=\left\{\alpha\in L^q([0,T],U)\middle|\int_0^T\left\| \alpha(t)\right\|_U^q d t \leq K\right\}.
    \end{equation*}
    Note, that all the results from section \ref{section:Well-Posedness of the Control Problem} still hold, if we have assumptions \ref{assumption:Linear operator}, \ref{Assumption:Standard Assumptions for Existence and Uniqueness of State Equation} and \ref{Assumption:f and g are bounded below and quadratically}. We do not need Assumption \ref{Assumptions:DifferentiabilityOfCoefficients}, \ref{Assumption: f,g diff. and moment bounds} and \ref{Assumption:Convexity of F,B,f} here, instead we additionally assume the following.
    \begin{assumption}[Compact embedding] \label{assumption:V compactly in H}
        $V$ embeds compactly into $H$.
    \end{assumption}
    \begin{assumption}[Continuity] \label{assumption:Condition for existence of controls}
        \begin{enumerate}[(i)]
            \item There exist $A_6>0$ and $1\leq r\leq q^\prime<q$, such that for for any $u(t),v(t)\in L^2([0,T],V)\cap L^q([0,T],H)$, $\mu,\nu:[0,T]\to \cP_q(H)$ measurable, $\alpha\in \A$ and $0\neq v\in V$
            \begin{align*}
                &\frac{1}{A_6\|v\|_V}\int_0^T \langle F(t,u_t,\mu_t,\alpha_t)-F(t,v_t,\nu_t,\alpha_t),v\rangle_Vdt\\
                \leq &
                \rho(u,v,\mu,\nu,\alpha)\int_0^T\left\|u_t-v_t\right\|^{r}_H+\W_{q^\prime}(\mu_t,\nu_t)dt
            \end{align*}
            and
            \begin{align*}
                &\frac{1}{A_6\|v\|_V}\int_0^T f(t,u_t,\mu_t,\alpha_t)-f(t,v_t,\nu_t,\alpha_t)dt\\
                \leq &
                \rho(u,v,\mu,\nu,\alpha)\int_0^T\left\|u_t-v_t\right\|^{r}_H+\W_{q^\prime}(\mu_t,\nu_t)dt
            \end{align*}
            where
            \begin{align*}
                \rho(u,v,\mu,\nu,\alpha)
                =&1+\int_0^T\|u_t\|^{q^\prime-r}_{H}+\|v_t\|^{q^\prime-r}_{H} +\mu(\|\cdot\|_H^{q})+\nu(\|\cdot\|_H^{q})
                +\|\alpha_t\|^q_Udt\\
                &+ \left(\int_0^T \|u_t\|^2_V+\|v_t\|^2_V dt\right)^{\frac{q^\prime-r}{2}}.
            \end{align*}
            \item For all $\varphi \in C\left([0, T] ; H\right)$, $\mu \in C\left([0, T] ; \mathcal{P}_2\left(H\right)\right)$ and $h\in H$
            \begin{align*}
                &\mathbb{A} \rightarrow L^2\left([0, T] , V^*\right),\quad \alpha \mapsto F\left(\cdot, \varphi_\cdot, \mu_\cdot, \alpha_\cdot\right),\\
                & \mathbb{A} \rightarrow L^2\left([0, T] ;H\right),\quad \alpha \mapsto (B^*B)\left(\cdot, \varphi_\cdot, \mu_\cdot, \alpha_\cdot\right)h,
            \end{align*}
            are weakly-weakly sequentially continuous and 
            \begin{equation*}
                \mathbb{A} \rightarrow \mathbb{R}, \quad \alpha \mapsto \int_0^T f\left(t, \varphi_t, \mu_t, \alpha_t\right) d t
            \end{equation*}
            is weakly sequentially lower semicontinuous.
            \item $g$ is lower semi-continuous in the following sense: If $x_n\to x$ weakly in $H$, $\mu_n\to \mu$ weakly in $\cP(H_{\text{weak}})$ and $\sup_{n\in\N}\mu_n\left(\|\cdot\|^q_H\right)<\infty$, then
            \begin{equation*}
                \liminf_{n\to\infty} g(x_n,\mu_n)\geq g(x,\mu).
            \end{equation*}
            Here, $H_{\text{weak}}$ denotes the Hilbert space $H$ equipped with its weak (in the functional analytic sense) topology and $\cP(H_{\text{weak}})$ are the probability measures on $H_\text{weak}$.
        \end{enumerate}
    \end{assumption} 
    \begin{remark}\label{remark:Sufficient conditions for existence}
        \begin{enumerate}
            \item This abstract assumption will be needed in the reaction-diffusion equation case (cf. section \ref{example:Reaction-diffusion-equation}). 
            \item As this assumption has quite a complicated structure, we give simpler sufficient conditions replacing the first and the third condition.
            \begin{enumerate}[(i)]
                \item[(i)$^\prime$] $F(t,x,\mu,\alpha)=F_1(t,x,\mu)+F_2(t,\alpha)$ and $f(t,x,\mu,\alpha)=f_1(t,x,\mu)+f_2(t,\alpha)$, where 
                \begin{equation*}
                    F_1(t,\cdot):H\times \cP_2(H) \to H\text{ and }f_1(t,\cdot):H\times \cP_2(H) \to \R
                \end{equation*}
                   are continuous.
                \item[(iii)$^\prime$] $g(x,\mu)=g_1(x)+g_2(\int_H y\, \mu(dy))$, where $g_1,g_2:H\to \R$ are continuous and convex.
            \end{enumerate}
        \end{enumerate}
    \end{remark}
    \begin{theorem} \label{theorem: existence of an optimal control}
        If Assumptions \ref{assumption:Linear operator}, \ref{Assumption:Standard Assumptions for Existence and Uniqueness of State Equation}, \ref{assumption:V compactly in H} and \ref{assumption:Condition for existence of controls} hold, there is at least one optimal control $\alpha^* \in \mathbb{A}$ such that $J\left(\alpha^*\right)=\inf _{\alpha \in \mathbb{A}} J(\alpha)$.
    \end{theorem}
    \begin{proof}
        We will again use our notation $\theta^\alpha_s=(s, X_s^\alpha, \cL(X_s^\alpha),\alpha_s)$. By Lemma \ref{lemma:well-posedness of J}, the control problem is finite and we can consider a sequence $\alpha^n\in\A$ such that $\lim _{n \rightarrow \infty} J\left(\alpha^n\right)=\inf _{\alpha \in \mathbb{A}} J(\alpha)$. 
        $\A \subset L^q\left([0, T],U\right)$ is bounded, thus, using the Banach-Alaoglu-Theorem, there exists an $\alpha^* \in L^q\left([0, T] ,U\right)$ and a subsequence, also denoted by $\left(\alpha^n\right)_{n \in \mathbb{N}}$, such that $\alpha^n \rightharpoonup \alpha^*$ weakly in $L^q\left([0, T] ,U\right)$.
        Since $\mathbb{A}$ is convex and (strongly) closed, it is also weakly closed and we get $\alpha^* \in \mathbb{A}$, so $\alpha^*$ is indeed an admissible control.\\
        Further, due to the a priori estimate in Lemma \ref{Lemma:a priori-Estimate}, we get
        \begin{equation}
            \sup_{n\in\N}\mathbb{E}\left[\sup _{t \in[0, T]}\left\|X^{\alpha^n}_t\right\|_{H}^q+\left(\int_0^T\left\|X^{\alpha^n}_t\right\|_{V}^2 \mathrm{~d} t\right)^{q/2}\right] <\infty. \label{eq:a priori-Estimate in Existence proof}
        \end{equation}
        We now show tightness of the laws of $X^{\alpha^n}$ in three different spaces. Namely $\Omega_1=L^2([0,T],V)_{\text{weak}}$, $\Omega_2=C([0,T],H_{\text{weak}})$ and $\Omega_3=L^2([0,T],H)$. Here $\Omega_2$ is equipped with the topology of uniform convergence, which is defined as $H_{\text{weak}}$ is a uniform space (cf. the Appendix \refforappendixtwo)\\
        \textbf{$\Omega_1$:} As $L^2([0,T],V)$ is a reflexive Banach space, using the Banach-Alaoglu theorem, it is clear from \eqref{eq:a priori-Estimate in Existence proof}, that the laws of $X^{\alpha^n}$ are tight on $L^2([0,T],V)$ equipped with the weak topology.\\
        \textbf{$\Omega_2$:}
        We have for $v\in V$ with $\|v\|_V\leq 1$, using Jensen, Assumption \ref{Assumption:Standard Assumptions for Existence and Uniqueness of State Equation} (H4) and \eqref{eq:a priori-Estimate in Existence proof}
        \begin{align*}
            \E\left[\left|\int_s^t \langle LX^\alpha_r+ F\left(\theta^\alpha_r\right), v\rangle_V dr\right|\right]
            \leq & \E\left[\left(t-s\right)^{\frac{1}{2}}\left(\int_0^T \| LX^\alpha_r + F\left(\theta^\alpha_r\right)\|^{2}_{V^*}\| v\|^2_V dr\right)^{{\frac{1}{2}}}\right]\\
            \leq & c_1\left(t-s\right)^{\frac{1}{2}}
        \end{align*}
        Further, using the BDG-inequality, Assumption \ref{Assumption:Standard Assumptions for Existence and Uniqueness of State Equation} (H3b), $q>2$ together with Jensen and \eqref{eq:a priori-Estimate in Existence proof}, we get
        \begin{align*}
            \E\left[\left|\langle \int_s^t B\left(\theta^\alpha_r\right) d W_r , v\rangle_H\right|\right] 
            \leq& c_2\left(t-s\right)^{\frac{q-2}{q}}.
        \end{align*}
        Putting these together, we arrive at
        \begin{align*}
            \E\left[\left|\langle X^\alpha_t-X_s^\alpha , v \rangle_H\right|\right]
            \leq  c_3(t-s)^{\tilde{q}}
        \end{align*}
        for some $\tilde{q}>0$.
        Now, we use the Arzela-Ascoli theorem (cf. Appendix \refforappendixone) to conclude tightness on $\Omega_2$.
        Note, that normally, one would need to show equicontinuity for every $h\in H$, but due to the uniform boundedness from \eqref{eq:a priori-Estimate in Existence proof}, it is sufficient to use the dense subset $V$.\\
        \textbf{$\Omega_3$:} Since we have already shown the equicontinuity, we exploit, using Assumption \ref{assumption:V compactly in H}, the following lemma.
         \begin{lemma}[Lemma 2.22 from \cite{Pardoux2021}]
            Given that the injection from $V$ into $H$ is compact, from any sequence $\left\{u_n, n \geq 1\right\}$ which is both bounded in $L^2([0, T],V) \cap L^\infty([0, T],H)$ and equicontinuous as $V^{\prime}$-valued functions, and such that the sequence $\left\{u_n(0)\right\}$ converges strongly in $H$, one can extract a subsequence which converges in $L^2([0, T],H)$ strongly.
        \end{lemma} \noindent
        Now, note the following:
        \begin{enumerate}[(I)]
            \item By the generalized Skorokhod theorem from \cite{Jakubowski1998} (cf. Appendix \refforappendixone) there exists a probability space $(\tilde{\Omega}, \tilde{\mathcal{F}}, \tilde{\mathbb{P}})$ and a sequence of random variables $(\tilde{X}^{\alpha^n})_{n \in \mathbb{N}}$ defined on $\tilde{\Omega}$ with the same law as $(X^{\alpha^n})_{n \in \mathbb{N}}$ and $\tilde{X}^{\alpha^*}$ defined on $\tilde{\Omega}$, such that the laws converge weakly (in the probabilistic sense) on $L^2([0,T],V)_{\text{weak}}$ and $C([0,T],H_{\text{weak}})$ and $\tilde{\mathbb{P}}$-almost surely we have $\tilde{X}^{\alpha^n} \rightarrow \tilde{X}^{\alpha^*}$ in $L^2([0, T],H)$.
            \item Due to the higher moments bound from \eqref{eq:a priori-Estimate in Existence proof} ($q>2$), we get the uniform integrability of $(\tilde{X}^{\alpha^n})_{n\in\N}$. This together with the a.s. convergence in $L^{2}([0,T],H)$ gives convergence in $L^{q^\prime}(\Omega,L^{q^\prime}([0,T],H))\simeq L^{q^\prime}([0,T],L^{q^\prime}(\Omega,H))\simeq L^{q^\prime}(\Omega\times[0,T],H)$ for all $q^\prime< q$. 
            \item In particular, along a subsequence we have $dt$-a.s. convergence in $L^{q^\prime}(\Omega,H)$ and hence also weak convergence of the laws on $H$. So, we also have $dt$-a.s. $\W_{q^\prime}(\tilde{X}^{\alpha^n}_t,\tilde{X}^{\alpha^*}_t)\to 0$ .
        \end{enumerate}
        It remains to identify the law of $\tilde{X}^{\alpha^*}$ as the law of the solution to \eqref{eq:TheBase-State-SPDE} with $\alpha^*$. To use a martingale representation theorem (cf. \cite{DaPrato2014}, \cite{GOLDYS2009}), we want to show, that, for all $s,t\in[0,T]$, $v\in V$ and $\varphi \in L^\infty(\mathcal{F}_s)$, we can replace $\alpha^n$ by $\alpha^*$ in
        \begin{equation}
            \widetilde{\mathbb{E}}\left[\langle \tilde{X}_t^{\alpha^n}-\tilde{X}_s^{\alpha^n}-\int_s^t L\tilde{X}^{\alpha^n}+ F(\tilde{\theta}^{\alpha^n}_r) d r , v \varphi \rangle_H\right]=0, \label{eq:DriftIdentification}
        \end{equation} 
        i.e. we can identify the drift part of our limiting process, and that, for all $s,t\in[0,T]$ and $a,b\in V$, we can also do the same replacement in
        \begin{equation}
            \begin{aligned}
                0=&\widetilde{\mathbb{E}}\left[ \langle\tilde{X}_t^{\alpha^n}-\tilde{X}_s- \int_s^t L\tilde{X}_r^{\alpha^n}+F(\tilde{\theta}^{\alpha^n}_r) d r, a\rangle_H \langle\tilde{X}_t^{\alpha^n}-\tilde{X}_s-\int_s^t L\tilde{X}_r^{\alpha^n}\right.\\
                & +\left.F(\tilde{\theta}^{\alpha^n}_r) dr, b\rangle_H \varphi-\int_s^t\langle  B(\tilde{\theta}^{\alpha^n}_r)  a, B(\tilde{\theta}^{\alpha^n}_r)  b\rangle_H d r \varphi\right],
            \end{aligned} \label{eq:MartingaleIdentification}
        \end{equation}
        i.e. we can show that the martingale part has the correct quadratic variation.\\
        Towards the convergence of \eqref{eq:DriftIdentification}, we consider the terms separately. The convergence of the $\langle\tilde{X}_t^{\alpha^n},\varphi v\rangle_H$-parts (for all $t\in[0,T]$) follows by the weak $\Omega_2$-convergence of the laws from (I) together uniform integrability from the bounds in \eqref{eq:a priori-Estimate in Existence proof}. The convergence of the $\langle \int LX^{\alpha^n},v\phi\rangle_V$-part follows by the weak $\Omega_1$-convergence of the laws from (I), using the continuity and linearity of $L:V\to V^*$ and uniform integrability from the bounds in \eqref{eq:a priori-Estimate in Existence proof}. Finally, the convergence of the $\langle \int F(\theta^{\alpha^n}),v\phi\rangle_V$-part follows by Assumption \ref{assumption:Condition for existence of controls} (i) and (ii) together with (II) and (III), using Assumption \ref{Assumption:Standard Assumptions for Existence and Uniqueness of State Equation} (H4) and \eqref{eq:a priori-Estimate in Existence proof} to get uniform integrability.\\
        Towards the convergence of \eqref{eq:MartingaleIdentification}, we do the same: By a similar argumentation as above, using that our a priori bounds hold for high enough moments ($q>p+2$) and uniform integrability, the convergence of the minuend follows (as we basically need $L^2$-convergence instead of $L^1$-convergence). And, as $B$ is Lipschitz by Assumption \ref{Assumption:Standard Assumptions for Existence and Uniqueness of State Equation} (H3b), we get similar uniform integrability, and, as $B^*B$ is continuous in the necessary way by Assumption \ref{assumption:Condition for existence of controls} (ii), the necessary convergence and can then use the same argumentation for the subtrahend, too.\\
        Now, we can use a martingale representation theorem (cf. \cite{DaPrato2014} Theorem 8.2) to get our desired optimal solution $\tilde{X}^\alpha$ (maybe on some other extended probability space) and, due to the strong existence and uniqueness of the solution to the state equation, we have $\mathbb{P} \circ(X^{\alpha^*})^{-1}=\tilde{\mathbb{P}}\circ(\tilde{X}^{\alpha^*})^{-1}$. Thus,
        \begin{align*}
            \inf _{\alpha \in \mathbb{A}} J(\alpha) 
            & =\liminf _{n \rightarrow \infty} \tilde{\mathbb{E}}\left[\int_0^T f\left(t, \tilde{X}_t ^{\alpha^n}, \cL\left(\tilde{X}_t ^{\alpha^n}\right), \alpha_t^n\right) d t+g\left(\tilde{X}_T ^{\alpha^n}, \cL\left(\tilde{X}_T ^{\alpha^n}\right)\right)\right] \\ 
            & \geq \tilde{\mathbb{E}} \left[\liminf _{n \rightarrow \infty} \int_0^T f\left(t, \tilde{X}_t ^{\alpha^n}, \cL\left(\tilde{X}_t ^{\alpha^n}\right), \alpha_t^n\right) d t+g\left(\tilde{X}_T ^{\alpha^n}, \cL\left(\tilde{X}_T ^{\alpha^n}\right)\right)\right] \\ 
            & \geq \tilde{\mathbb{E}}\left[\int_0^T f\left(t, \tilde{X}^{\alpha^*}_t, \cL\left(\tilde{X}^{\alpha^*}_t\right), \alpha_t\right) d t+g\left(\tilde{X}^{\alpha^*}_T, \cL\left(\tilde{X}^{\alpha^*}_T\right)\right)\right] \\ 
            & =J(\alpha^*),
        \end{align*}
        where 
        in the first inequality we use Fatou's Lemma, possible due to the lower boundedness of $f$ and $g$ from Assumption \ref{Assumption:f and g are bounded below and quadratically} and in the second inequality, we use the weak sequential lower semi-continuity of $f$ from Assumption \ref{assumption:Condition for existence of controls} (i) and lower semi-continuity of $g$ from Assumption \ref{assumption:Condition for existence of controls} (iii).
    \end{proof}

\section{Examples}\label{section:examples}
\subsection{Linear-Quadratic Case}
    In the setting introduced in section \ref{section:Preliminaries and Assumptions}, we consider the following standard linear-quadratic case,
    \begin{itemize}
        \item $F(t,x,\mu,\alpha)=F_1(t)x+F_2(t) \int_H y \,\mu(dy)+F_3(t)\alpha$,
        \item $B(t,x,\mu,\alpha)=B_1(t)x+B_2(t)\int_H y\,\mu(dy)+B_3(t)\alpha$,
    \end{itemize}
    the cost functions are
    \begin{itemize}
        \item $f(t,x,\mu,\alpha)=\langle x, f_1(t)x\rangle_H+\langle x-h_1(t)\int_H y\,\mu(dy),f_2(t)(x-h_1(t)\int_H y\,\mu(dy)\rangle_H$ $+\langle \alpha, f_3(t)\alpha\rangle_U$,
        \item $g(x,\mu)=\langle x, g_1 x\rangle_H+\langle x-h_2\int_H y\,\mu(dy), g_2\left(x-h_2\int_H y\,\mu(dy)\right)\rangle_H$
    \end{itemize}
    and we choose $q>2$. \\
    We denote $L(E_1,E_2)$ the linear bounded operators from some normed spaces $E_1$ to $E_2$ equipped with the operator norm and $L_{s,p}(E_1)$ the set of symmetric, nonnegative semi-definite, linear, bounded operators on $E_1$. Let $L$ fulfill Assumption \ref{assumption:Linear operator} and let
    \begin{itemize}
        \item $F_1, F_2 :[0,T]\mapsto L(H,H)$, $F_3:[0,T]\mapsto L(U,H)$,
        \item $B_1, B_2 :[0,T]\mapsto L(H,L_2(H))$, $B_3:[0,T]\mapsto L(U,L_2(H))$,
        \item $f_1, f_2 :[0,T]\mapsto L_{s,p}(H)$, $f_3:[0,T]\mapsto L_{s,p}(U)$,
        \item $h_1, h_2 :[0,T]\mapsto L(H,H)$,
        \item $g_1, g_2 \in L_{s,p}(H,H)$,
    \end{itemize}
    be measurable and all their operator norms be uniformly bounded in $t$.\\
    Then, Assumption \ref{Assumption:Standard Assumptions for Existence and Uniqueness of State Equation},
    \ref{Assumption:f and g are bounded below and quadratically},
    \ref{Assumptions:DifferentiabilityOfCoefficients}
    are fulfilled and $p$ from \ref{Assumption:Standard Assumptions for Existence and Uniqueness of State Equation} and $p^\prime$ from \ref{Assumptions:DifferentiabilityOfCoefficients} are $0$ in this case, so $q>2$ is indeed sufficient.
    Assumption \ref{Assumption: f,g diff. and moment bounds} 
    and \ref{Assumption:Convexity of F,B,f}
    are also easy to check and thus, our results until Theorem \ref{theorem:Pontryagin Maximum Principle} are applicable to this case.\\
    For the existence of a deterministic optimal control, we additionally only need to assume that $V$ embeds compactly into $H$. Assumption \ref{assumption:Condition for existence of controls} is fulfilled, as bounded, linear maps are always weakly-weakly continuous and $f$ and $g$ are weakly sequentially lower-semicontinuous due to their continuity and convexity (cf. Remark \ref{remark:Sufficient conditions for existence}).
\subsection{Reaction-Diffusion Equation} \label{example:Reaction-diffusion-equation}
    We will look at a generalization of \cite{StannatWessels2021}.
    Let $\Delta$ be the Dirichlet-Laplace on some bounded Lipschitz domain $D\subset \R^d$, with $d\leq 3$, and $V=H_0^1(D,\R)\subset L^2(D,\R)= H$ and let $U=H$.
    Consider the equation
    \begin{equation*}
        dX_t= \left(\Delta X_t + F_1(X_t)+F_2(\cL(X_t))+b_t\alpha_t \right)dt + B dW_t,
    \end{equation*}
    where $F_1:\R\to\R$ acts as a Nemytskii operator, $F_2:\cP_2(H)\to H$ is Lipschitz
    and $\Lambda$-continuously $L$-differentiable, $b_t\in L^\infty ([0,T]\times D)$, $B\in L_2(H)$, $W_t$ a cylindrical Wiener process on $H$. \\
    We make the following assumptions depending on the dimension of the domain: If $d=1$, then $r=3$ and $q>6$, if $d=2$, then $r<3$ and $q>6$ and if $d=3$, then $r=\frac{7}{3}$ and $q>16/3$.\\
    Now, we assume, that for all $x\in\R$ it holds
    \begin{enumerate}[(i)]
        \item $F_1(x)x \leq C$,
        \item $|F_1(x)|\leq C(1+|x|^r)$,
        \item $F_1:\R\to\R$ is continuously differentiable,
        \item $\sup _{y \in \mathbb{R}} F_1^{\prime}(y)<\infty$ and
        \item $\left|F_1^{\prime}(x)\right|<C\left(1+|x|^{r-1}\right)$.
    \end{enumerate}
    Note, that the upper bound on the derivative implies monotonicity,
    \begin{equation*}
        (F_1(x)-F_1(y))(x-y) \leq C|x-y|^2\quad\text{for all }x,y\in \R
    \end{equation*}
    and
    \begin{equation*}
        (F^\prime_1(x)z_1-F^\prime_1(x)z_2)(z_1-z_2) \leq C|z_1-z_2|^2\quad\text{for all }x,z_1,z_3\in \R.
    \end{equation*}
    We want to optimize with respect to the quadratic cost functional
    \begin{align*}
        J(\alpha)
        :=&\mathbb{E}\left[\int_0^T \|X^\alpha_t-\bar{u}_t\|_H^2+\|\alpha_t-\bar{\alpha}_t\|_H^2 d t+ \|X^\alpha_T - \bar{u}_T\|_H^2\right],
    \end{align*}
    where $\bar{u},\bar{\alpha}\in L^2([0,T],H)$ and $\bar{u}_T\in H$ are running and terminal reference profiles.\\
    Now, Assumptions \ref{assumption:Linear operator}, \ref{Assumption:Standard Assumptions for Existence and Uniqueness of State Equation}, \ref{Assumption:f and g are bounded below and quadratically}, \ref{Assumptions:DifferentiabilityOfCoefficients}, \ref{Assumption: f,g diff. and moment bounds} and \ref{Assumption:Convexity of F,B,f} hold true, so all the results until the Pontryagin maximum principle from Theorem \ref{theorem:Pontryagin Maximum Principle} holds.\\
    For the existence of optimal controls, we let the controls be deterministic and restrict to $d\leq 2$. We see, that Assumption \ref{assumption:V compactly in H} holds. Since Assumption \ref{assumption:Condition for existence of controls} is so specific, we look at the explicit example $F_1(x)=-x^3$.
    We notice, that $F_2$ is Lipschitz continuous and $\alpha_t\to b_t\alpha_t$ is linear and bounded and therefore weakly-weakly continuous (cf. Remark \ref{remark:Sufficient conditions for existence}). Towards the bound of $F_1$, we notice, that 
    \begin{equation*}
        x^3- y^3=(x- y)\left(x^2+y^2+xy\right).
    \end{equation*}
    Thus, using the Hölder-inequality and interpolation inequalities (cf. \cite{adams2003sobolev} 5.8), we get
    \begin{align*}
        &\int_0^T \langle u(t)^3- u_n(t)^3,v\rangle_V dt
        \leq  \|v\|_{L^{r^\prime}} \int_0^T \left\|u(t)^3- u_n(t)^3\right\|_{L^{\frac{r^\prime}{r^\prime-1}}}dt\\
        \leq & c_1\|v\|_{V} \int_0^T   \left(\|u(t)\|^2_{L^{r^\prime}}+\|u_n(t)\|^2_{L^{r^\prime}}\right) 
        \left\|(u(t)- u_n(t))\right\|_{L^{\frac{r^\prime}{r^\prime-3}}}dt\\
        \leq & c_2\|v\|_V \left(\int_0^T \|u(t)\|^{\frac{2\theta}{\gamma}}_V\|u(t)\|^{\frac{2(1-\theta)}{\gamma}}_H
        +\|u_n(t)\|^{\frac{2\theta}{\gamma}}_V\|u_n(t)\|^{\frac{2(1-\theta)}{\gamma}}_H dt\right)^{\gamma}\\
        &\cdot \left(\int_0^T \left\|(u(t)- u_n(t))\right\|_V^{\frac{\theta^\prime}{1-\gamma}}\left\|(u(t)- u_n(t))\right\|_H^{\frac{1-\theta^\prime}{1-\gamma}} dt\right)^{1-\gamma}
    \end{align*}
    for some $0<\gamma<1$ and $\theta = \frac{d}{2}-\frac{d}{r^\prime}$ and $\theta^\prime = \frac{d}{2}-\frac{d(r^\prime-3)}{r^\prime}.$
    Thus, choosing $r^\prime=4$, we get
    \begin{equation*}
        \theta
        =\theta^\prime
        =\begin{cases}
            \frac{1}{4} & d=1\\
            \frac{1}{2} & d=2
        \end{cases}    
    \end{equation*}
    and, choosing $\gamma=\frac{2}{3}$ and using the Hölder-inequality once more, we get the desired bound. 
    Towards the lower semi-continuity of $f$, we notice, that $f$ is continuous and convex. This immediately gives the weak sequential lower semi-continuity. As $B$ is constant, it also fulfills the continuity from Assumption \ref{assumption:Condition for existence of controls} (ii) and $g$ is also convex and continuous, so it also fulfills the desired continuity from (iii).

\section*{Acknowledgments}
WS acknowledges support from DFG CRC/TRR 388 'Rough Analysis, Stochastic Dynamics and Related Fields', Projects A10 and B09.
\appendix
\section{Topologies and Compactness on General Function Spaces}\label{appendix:topologies on general function spaces}
    For our purposes, we first need to define some topologies on spaces of functions. For example, as a Hilbert space $H$ with its weak topology $\cH_\sigma$ is not a metric space, we cannot use the standard supremum topology on the space of continuous functions $C([0,T],(H,\cH_\sigma))$. But, using the theory of uniform spaces, one can define a similar topology, the so-called topology of uniform convergence. 
    We do not introduce uniform spaces completely here, but refer to e.g. \cite{willard2004general} Chapter $9$ for an introduction. Relevant for our result will be the following characterization.
    \begin{definition}[Gauge space \cite{SCHECHTER1997} 5.15 h]
        \begin{enumerate}[(i)]
            \item A collection of pseudometrics $D$ on a set $X$ is called a gauge. 
            \item For each $d \in D$ let $B_d$ be the corresponding open ball. Let $\mathcal{T}_D$ be the collection of all sets $T \subseteq X$ having the property that
            \begin{equation*}
            \begin{aligned}
            & \text { for each } x \in T \text {, there is some finite set } D_0 \subseteq D \text { and some number } r>0 \text { such } \\
            & \text { that } \bigcap_{d \in D_0} B_d(x, r) \subseteq T \text {. }
            \end{aligned}
            \end{equation*}
            Then $\mathcal{T}_D$ is a topology on $X$. We will call it the gauge topology determined by $D$. Any gauge space $(X, D)$ will be understood to be equipped with this topology.
        \end{enumerate}
    \end{definition}
    We get the following result.
    \begin{theorem} \label{theorem:equivalence uniform, gauge and completely regular}
        Let $(X, \mathcal{T})$ be a topological space. Then the following conditions are equivalent:
         \begin{enumerate} [(A)]
             \item $\mathcal{T}$ is gaugeable - i.e., given by a gauge, in the sense of the above definition.
             \item $\mathcal{T}$ is uniformizable - i.e., given by a uniformity.
             \item $\mathcal{T}$ is completely regular.
         \end{enumerate}
    \end{theorem}
    Other authors also call a collection of pseudometrics that is equivalent to a uniformity a gage, if it contains all the uniformly continuous (w.r.t. the product uniformity on $X\times X$) pseudometrics. (cf. \cite{kelley1975} in Chapter 'Uniform spaces' after Corollary 17).\\
    \begin{definition}\label{definition:initial topology}
        Let $X$ be a space and $(Y_\lambda,\cY_\lambda)_{\lambda\in\Lambda}$ a family of topological spaces. For a family of functions $\cF=\{\varphi_\lambda:X\to(Y_\lambda,\cY_\lambda)\mid\lambda\in\Lambda\}$, we denote $\sigma\left(\cF\right)$ the smallest topology on $X$ such that every function in $\cF$ is continuous. This is called the initial topology induced by $\cF$.
    \end{definition} 
    Conveniently, we have the following result.
    \begin{theorem}[\cite{SCHECHTER1997} 18.9 g] \label{theorem:initial topologies are uniform topologies}
        \begin{enumerate}[(i)]
            \item The initial topology determined directly by mappings in $\cF$ is the same, as the uniform topology determined by the initial uniformity, that is generated by $\cF$.
            \item In particular, every initial topology is a uniform topology.
        \end{enumerate}
    \end{theorem}
    In fact, the uniform topologies are exactly those, that are initial topologies (cf. \cite{Aliprantis2006} Corollary 2.56).
    Further, we also get the following.
    \begin{theorem}[\cite{SCHECHTER1997} 18.9 f]\label{theorem:initial gauge}
        Let $X$ be a set, let $\left\{\left(Y_\lambda, E_\lambda\right): \lambda \in \Lambda\right\}$ be a collection of gauge spaces, and let $\varphi_\lambda$ : $X \rightarrow Y_\lambda$ be some mappings. Then, the initial uniformity on $X$ determined by the $\varphi_\lambda$ and $E_\lambda$ is equal to the uniformity on $X$ determined by the gauge $D=\bigcup_{\lambda \in \Lambda}\left\{e \varphi_\lambda: e \in E_\lambda\right\}$, where $\left(e \varphi_\lambda\right)\left(x, x^{\prime}\right)=e\left(\varphi_\lambda(x), \varphi_\lambda\left(x^{\prime}\right)\right)$.
    \end{theorem}
    We now introduce topologies on function spaces.
    \begin{definition}[Uniform convergence on members of $\cS$ (\cite{SCHECHTER1997} 18.26)] \label{definition: topology of uniform convergence on S}
        Let $X$ be a set, and let $\mathcal{S}$ be a collection of subsets of $X$. Let $(Y, \mathcal{U})$ be a uniform space. A net $\left(\varphi_\alpha\right)$ in $Y^X=\{$functions from $X$ into $Y\}$ will be said to converge $\mathcal{U}$-uniformly on elements of $\mathcal{S}$ to a limit $\varphi \in Y^X$ if
        for each $S \in \mathcal{S}$ and $U \in \mathcal{U}$, eventually $\left\{\left(\varphi_\alpha(s), \varphi(s)\right): s \in S\right\} \subseteq U$.\\
        This can be expressed in terms of gauges as well. Let $E$ be any gauge on $Y$ that yields the uniformity $\mathcal{U}$; then $\varphi_\alpha \rightarrow \varphi$ in the sense above if and only if
        \begin{equation*}
            \sup _{s \in S} e\left(\varphi_\alpha(s), \varphi(s)\right) \quad \rightarrow \quad 0 \quad \text { for each } e \in E \text { and } S \in \mathcal{S}
        \end{equation*}
    \end{definition}
    \begin{theorem}[\cite{SCHECHTER1997} p.492] \label{theorem:topology of uniform convergence is gaugeable and a uniform topology}
        The above defined convergence of uniform convergence on members of $\cS$ is topological and gaugeable by the gauge
        \begin{equation*}
            \left\{\sup _{s \in S} e^{\prime}(\cdot(s), \cdot(s))\middle|S\in\cS,e^\prime\in E^\prime\right\},
        \end{equation*}
        where $E^\prime$ is a bounded gauge, that is uniformly equivalent to $E$ (this always exists (cf. \cite{SCHECHTER1997} 18.14)).\\
        In particular this topology is a uniform topology.
    \end{theorem}
    \begin{definition}[\cite{SCHECHTER1997} p.492]
        In the setting of the above Definition \ref{definition: topology of uniform convergence on S} we have the following special cases.
        \begin{enumerate}[(i)]
            \item If $\mathcal{S}$ contains just the singletons $\{x\}$ (for $x \in X$ ), then uniform convergence on elements of $\mathcal{S}$ is the same thing as pointwise convergence. Thus, the product topology on $Y^X$ is a special case of uniform convergence topologies.
            \item When $\mathcal{S}=\{X\}$, then the uniform convergence topology is called simply the topology of uniform convergence on $X$.
            \item When $X$ is a topological space, another important choice is $\mathcal{S}=\{$compact subsets of $X\}$, resulting in the topology of uniform convergence on compact sets.
        \end{enumerate}
    \end{definition}
    \begin{remark}\label{remark:X compact implies equality of uniform topologies}
        Clearly, if $X\in \cS$, then $\cS$ can be reduced to $\{X\}$ and it results in the topology of uniform convergence. So if $X$ is compact, then the topology of uniform convergence and the topology of uniform convergence on compact sets coincide.
    \end{remark}
    Finally, we want to characterize the compact sets by a generalized Arzelà-Ascoli theorem. For this, we first need to introduce equicontinuity.
    \begin{definition}[Equicontinuity (cf. \cite{SCHECHTER1997} 18.29)] \label{definition:Equicontinuity on Uniform spaces}
        Let $X$ be a topological space, and let $x_0 \in X$. Let $(Y, \mathcal{V})$ be a uniform space, with uniformity determined by a gauge $E$. Let $\Phi$ be a collection of functions from $X$ into $Y$. We shall say that $\Phi$ is equicontinuous at the point $x_0$, if for each $e \in E$ and $\varepsilon>0$, there is some neighborhood $H$ of $x_0$ in $X$ such that
        \begin{equation*}
            x \in H, \quad \varphi \in \Phi \quad \Rightarrow \quad e\left(\varphi(x), \varphi\left(x_0\right)\right) \leq \varepsilon .
        \end{equation*}
        A collection of mappings $\Phi: X \rightarrow Y$ is said to be equicontinuous if it is equicontinuous at every point of $X$.
    \end{definition}
    \begin{theorem}[Arzelá-Ascoli (\cite{SCHECHTER1997} 18.35)]\label{theorem:Arzela-Ascoli for general uniform spaces}
        Let $X$ be a topological space, let $Y$ be a uniform space, and let $C(X, Y)=\{$continuous functions from $X$ into $Y\}$ be given the topology of uniform convergence on compact subsets of $X$. Let $\Phi \subseteq C(X, Y)$. If
        \begin{enumerate}
            \item[(1a)] $\Phi$ is equicontinuous, and
            \item[(1b)] the set $\Phi(x)=\{\varphi(x): \varphi \in \Phi\}$ is relatively compact in $Y$ for each $x \in X$
        \end{enumerate}
        then
        \begin{enumerate}
            \item[(2)] $\Phi$ is relatively compact in $C(X, Y)$.
        \end{enumerate}
        If $X$ is locally compact or first countable; then (2) also implies (1).
    \end{theorem}
\section{Generalized Skorokhod Embedding and Its Application} \label{appendix:General Skorokhod and Application}
    \subsection{Generalized Skorokhod Embedding Theorem}
        We assume, that $(\mathcal{X}, \tau)$ is a topological space with the only assumption, that
        \begin{equation}
            \begin{aligned}
                &\text{there exists a countable family $\left\{f_i: \mathcal{X} \rightarrow[-1,1]\right\}_{i \in I}$}\\
                &\text{of $\tau$-continuous functions, which separate points of $\mathcal{X}$.}
            \end{aligned} \label{eq:condition from Jakubowski}
        \end{equation}
        Under this condition, \cite{Jakubowski1998} proved, that the following version of the Skorokhod embedding theorem holds.
        \begin{theorem}\label{Theorem:JakubowskisSkorokhodEmbedding}
            Let $(\mathcal{X}, \tau)$ be a topological space satisfying \eqref{eq:condition from Jakubowski} and let $(X_i)_{i\in\N}$ be $\mathcal{X}$-valued random variables. Suppose for each $\varepsilon>0$ there exists a compact set $K_{\varepsilon} \subset \mathcal{X}$ such that for all $n\in\N$
            \begin{equation*}
                \P\left(X_n \in K_{\varepsilon}\right)>1-\varepsilon,
            \end{equation*}
            Then one can find a subsequence $(X_{n_k})_{k \in \mathbb{N}}$ and $\mathcal{X}$-valued random variables $(Y_k)_{k\in\N_0}$ defined on $([0,1], \mathcal{B}_{[0,1]},dt)$ such that for all $k\in\N$
            \begin{equation*}
                X_{n_k} \sim Y_k,
            \end{equation*}
            and for all $\omega\in[0,1]$ and $k\to\infty$
            \begin{equation*}
                Y_k(\omega) \longrightarrow_{\tau} Y_0(\omega).
            \end{equation*}
        \end{theorem}
    \subsection{Application}
        We need to precisely define our spaces and their topologies, we equip $\Omega_1=L^2([0,T],V)$ with its weak topology
        \begin{equation*}
            \tau_1=\sigma\left(\{f:L^2([0,T],V)\to\R\middle|f\text{ linear and bounded}\}\right),
        \end{equation*}
        $\Omega_2=C([0,T],H_{\text{weak}})$ with the topology of of uniform convergence from Definition \ref{definition: topology of uniform convergence on S} denoted $\tau_2$, which is defined as $H$ equipped with its weak topology is a uniform space by Theorem \ref{theorem:initial topologies are uniform topologies}, and $\Omega_3=L^2([0,T],H)$ with its norm topology denoted $\tau_3$.\\
        For $\Omega_1$ and $\Omega_3$ the property \eqref{eq:condition from Jakubowski} is clear (cf. also \cite{Jakubowski1998} Theorem 4). For $\Omega_2$ we notice, that for $(t_n)_{n\in\N}$ dense in $[0,T]$ and $(h_n)_{n\in\N}$ dense in $H_{weak}$
        \begin{equation*}
            \{f_{n,m}:\Omega_2\to [-1,1],g\mapsto 1\wedge\langle h_n, g(t_n)\rangle_H\vee-1\mid n,m\in \N \}
        \end{equation*}
        is a countable separating family of continuous functions (point evaluations are continuous) of $\Omega_2$. Indeed, if $g_1(t)\neq g_2(t)$ the continuity and uniqueness of weak limits implies that there is an $N\in \N$ such that $g_1(t_N)\neq g_2(t_N)$ and by the density there is an $M\in \N$ such that $\langle h_M, g_1(t_N)\rangle_H\neq \langle h_M, g_2(t_N)\rangle_H$.

\setcitestyle{numbers}
\bibliographystyle{unsrtnat}
\bibliography{refs}

\begin{thebibliography}{36}
\providecommand{\natexlab}[1]{#1}
\providecommand{\url}[1]{\texttt{#1}}
\expandafter\ifx\csname urlstyle\endcsname\relax
  \providecommand{\doi}[1]{doi: #1}\else
  \providecommand{\doi}{doi: \begingroup \urlstyle{rm}\Url}\fi

\bibitem[Vogler and Stannat(2025)]{vogler2024lionsderivativeinfinitedimensions}
Alexander Vogler and Wilhelm Stannat.
\newblock The lions derivative in infinite dimensions -- application to higher order expansion of mean-field spdes, 2025.

\bibitem[Bismut(1978)]{Bismut1978}
Jean-Michel Bismut.
\newblock An introductory approach to duality in optimal stochastic control.
\newblock \emph{SIAM Review}, 20\penalty0 (1):\penalty0 62--78, 1978.
\newblock ISSN 00361445, 10957200.

\bibitem[Peng(1990)]{Peng1990}
Shige Peng.
\newblock A general stochastic maximum principle for optimal control problems.
\newblock \emph{SIAM Journal on Control and Optimization}, 28\penalty0 (4):\penalty0 966–979, 1990.
\newblock \doi{10.1137/0328054}.

\bibitem[Lasry and Lions(2007)]{LasryLions2007}
Jean-Michel Lasry and Pierre-Louis Lions.
\newblock Mean field games.
\newblock \emph{Japanese Journal of Mathematics}, 2\penalty0 (1):\penalty0 229--260, Mar 2007.
\newblock ISSN 1861-3624.
\newblock \doi{10.1007/s11537-007-0657-8}.

\bibitem[Carmona and Delarue(2018)]{CarmonaDelarue2018}
René Carmona and François Delarue.
\newblock \emph{Probabilistic Theory of Mean Field Games with Applications I: Mean Field FBSDEs, Control, and Games}.
\newblock Springer International Publishing, Cham, 2018.
\newblock ISBN 978-3-319-58920-6.
\newblock \doi{10.1007/978-3-319-58920-6}.

\bibitem[Buckdahn et~al.(2014)Buckdahn, Li, Shi-Ge, and Rainer]{Buckdahn2014}
Rainer Buckdahn, Juan Li, Peng Shi-Ge, and Catherine Rainer.
\newblock Mean-field stochastic differential equations and associated pdes.
\newblock \emph{The Annals of Probability}, 45, 07 2014.
\newblock \doi{10.1214/15-AOP1076}.

\bibitem[Pham and Wei(2017)]{PhamWei2020}
Huyên Pham and Xiaoli Wei.
\newblock Dynamic programming for optimal control of stochastic mckean–vlasov dynamics.
\newblock \emph{SIAM Journal on Control and Optimization}, 55\penalty0 (2):\penalty0 1069–1101, 2017.
\newblock \doi{10.1137/16M1071390}.

\bibitem[Da~Prato and Zabczyk(2014)]{DaPrato2014}
Giuseppe Da~Prato and Jerzy Zabczyk.
\newblock \emph{Stochastic Equations in Infinite Dimensions}.
\newblock Encyclopedia of Mathematics and its Applications. Cambridge University Press, 2 edition, 2014.
\newblock \doi{10.1017/CBO9781107295513}.

\bibitem[Liu and R{\"o}ckner(2015)]{Liu2015}
Wei Liu and Michael R{\"o}ckner.
\newblock \emph{Stochastic Partial Differential Equations: An Introduction}.
\newblock Springer International Publishing, 2015.
\newblock ISBN 978-3-319-22354-4.
\newblock \doi{10.1007/978-3-319-22354-4}.

\bibitem[Pardoux(2021)]{Pardoux2021}
{\'E}tienne Pardoux.
\newblock \emph{SPDEs as Infinite-Dimensional SDEs}, pages 9--39.
\newblock Springer International Publishing, Cham, 2021.
\newblock ISBN 978-3-030-89003-2.
\newblock \doi{10.1007/978-3-030-89003-2_2}.

\bibitem[Bensoussan(1983)]{BENSOUSSAN1983}
A.~Bensoussan.
\newblock Stochastic maximum principle for distributed parameter systems.
\newblock \emph{Journal of the Franklin Institute}, 315\penalty0 (5):\penalty0 387--406, 1983.
\newblock ISSN 0016-0032.
\newblock \doi{10.1016/0016-0032(83)90059-5}.

\bibitem[Fuhrman et~al.(2012)Fuhrman, Hu, and Tessitore]{Fuhrman2013}
Marco Fuhrman, Ying Hu, and Gianmario Tessitore.
\newblock Stochastic maximum principle for optimal control of spdes.
\newblock \emph{Comptes Rendus Mathematique}, 350\penalty0 (13):\penalty0 683–688, 2012.
\newblock ISSN 1631-073X.
\newblock \doi{10.1016/j.crma.2012.07.009}.

\bibitem[Lü and Zhang(2021)]{Lu2021}
Qi~Lü and Xu~Zhang.
\newblock \emph{Pontryagin-Type Stochastic Maximum Principle and Beyond}, page 387–475.
\newblock Springer International Publishing, Cham, 2021.
\newblock ISBN 978-3-030-82331-3.
\newblock \doi{10.1007/978-3-030-82331-3_12}.

\bibitem[Ahmed(2016)]{Ahmed2016}
N.U. Ahmed.
\newblock A general class of mckean-vlasov stochastic evolution equations driven by brownian motion and lèvy process and controlled by lèvy measure.
\newblock \emph{Discussiones Mathematicae, Differential Inclusions, Control and Optimization}, 36\penalty0 (2):\penalty0 181--206, 2016.

\bibitem[Dumitrescu et~al.(2018)Dumitrescu, {\O}ksendal, and Sulem]{Dumitrescu2018}
Roxana Dumitrescu, Bernt {\O}ksendal, and Agn{\`e}s Sulem.
\newblock Stochastic control for mean-field stochastic partial differential equations with jumps.
\newblock \emph{Journal of Optimization Theory and Applications}, 176\penalty0 (3):\penalty0 559--584, Mar 2018.
\newblock ISSN 1573-2878.
\newblock \doi{10.1007/s10957-018-1243-3}.

\bibitem[Tang et~al.(2019)Tang, Meng, and Wang]{Tang2019}
Maoning Tang, Qingxin Meng, and Meijiao Wang.
\newblock Forward and backward mean-field stochastic partial differential equation and optimal control.
\newblock \emph{Chinese Annals of Mathematics, Series B}, 40\penalty0 (4):\penalty0 515--540, Jul 2019.
\newblock ISSN 1860-6261.
\newblock \doi{10.1007/s11401-019-0149-1}.

\bibitem[Cosso et~al.(2023)Cosso, Gozzi, Kharroubi, Pham, and Rosestolato]{CossoGozziFausto2023}
Andrea Cosso, Fausto Gozzi, Idris Kharroubi, Huyên Pham, and Mauro Rosestolato.
\newblock Optimal control of path-dependent mckean-vlasov sdes in infinite dimension.
\newblock \emph{The Annals of Applied Probability}, 33:\penalty0 2863--2918, 08 2023.
\newblock \doi{10.1214/22-AAP1880}.

\bibitem[Gao et~al.(2022)Gao, Hong, and Liu]{Gao_2022}
Jingyue Gao, Wei Hong, and Wei Liu.
\newblock Distribution-dependent stochastic porous media equations.
\newblock \emph{Stochastics and Dynamics}, 22, 11 2022.
\newblock \doi{10.1142/S0219493722400263}.

\bibitem[Hong et~al.(2024)Hong, Hu, and Liu]{HongHuLiu2022}
Wei Hong, Shanshan Hu, and Wei Liu.
\newblock Mckean-vlasov sde and spde with locally monotone coefficients.
\newblock \emph{The Annals of Applied Probability}, 34:\penalty0 2136--2189, 04 2024.
\newblock \doi{10.1214/23-AAP2016}.

\bibitem[Stannat and Wessels(2021)]{StannatWessels2021}
Wilhelm Stannat and Lukas Wessels.
\newblock Deterministic control of stochastic reaction-diffusion equations.
\newblock \emph{Evolution Equations and Control Theory}, 10\penalty0 (4):\penalty0 701--722, 2021.
\newblock \doi{10.3934/eect.2020087}.

\bibitem[Hocquet and Vogler(2021)]{HocquetVogler2020}
Antoine Hocquet and Alexander Vogler.
\newblock Optimal control of mean field equations with monotone coefficients and applications in neuroscience.
\newblock \emph{Applied Mathematics {\&} Optimization}, 84\penalty0 (2):\penalty0 1925--1968, Dec 2021.
\newblock ISSN 1432-0606.
\newblock \doi{10.1007/s00245-021-09816-1}.

\bibitem[Villani(2009)]{Villani2009}
C{\'e}dric Villani.
\newblock \emph{The Wasserstein distances}, pages 93--111.
\newblock Springer Berlin Heidelberg, Berlin, Heidelberg, 2009.
\newblock ISBN 978-3-540-71050-9.
\newblock \doi{10.1007/978-3-540-71050-9_6}.

\bibitem[Ahmed(2013)]{Ahmed2013}
N.U. Ahmed.
\newblock A note on radon-nikodym theorem for operator valued measures and its applications.
\newblock \emph{Communications of the Korean Mathematical Society}, 28, 04 2013.
\newblock \doi{10.4134/CKMS.2013.28.2.285}.

\bibitem[Dinculeanu(1967)]{dinculeanu1967vector}
N.~Dinculeanu.
\newblock \emph{Vector Measures}.
\newblock Hochschulb{\"u}cher f{\"u}r Mathematik. Elsevier Science \& Technology Books, 1967.
\newblock ISBN 9780080121925.

\bibitem[Diestel and Uhl(1977)]{diestel1977vector}
J.~Diestel and J.J. Uhl.
\newblock \emph{Vector Measures}.
\newblock Mathematical surveys and monographs. American Mathematical Society, 1977.
\newblock ISBN 9780821815151.

\bibitem[Pan(2023)]{Pan2023}
Zigang Pan.
\newblock \emph{Measure-Theoretic Calculus in Abstract Spaces: On the Playground of Infinite-Dimensional Spaces}.
\newblock Springer International Publishing, Cham, 2023.
\newblock ISBN 978-3-031-21912-2.
\newblock \doi{10.1007/978-3-031-21912-2}.

\bibitem[Blackwell and Dubins(1983)]{BlackwellDubinsSkorokhodExstension1983}
David Blackwell and Lester~E. Dubins.
\newblock An extension of {S}korohod's almost sure representation theorem.
\newblock \emph{Proc. Amer. Math. Soc.}, 89\penalty0 (4):\penalty0 691--692, 1983.
\newblock ISSN 0002-9939,1088-6826.
\newblock \doi{10.2307/2044607}.

\bibitem[Pardoux and R{\u{a}}{\c{s}}canu(2014)]{Pardoux2014}
Etienne Pardoux and Aurel R{\u{a}}{\c{s}}canu.
\newblock \emph{Backward Stochastic Differential Equations}, pages 353--515.
\newblock Springer International Publishing, Cham, 2014.
\newblock ISBN 978-3-319-05714-9.
\newblock \doi{10.1007/978-3-319-05714-9_5}.

\bibitem[Fan et~al.(2011)Fan, Jiang, and Tian]{FAN2011427}
ShengJun Fan, Long Jiang, and DeJian Tian.
\newblock One-dimensional bsdes with finite and infinite time horizons.
\newblock \emph{Stochastic Processes and their Applications}, 121\penalty0 (3):\penalty0 427--440, 2011.
\newblock ISSN 0304-4149.
\newblock \doi{10.1016/j.spa.2010.11.008}.

\bibitem[Jakubowski(1998)]{Jakubowski1998}
A.~Jakubowski.
\newblock Short communication:the almost sure skorokhod representation for subsequences in nonmetric spaces.
\newblock \emph{Theory of Probability \& Its Applications}, 42\penalty0 (1):\penalty0 167--174, 1998.
\newblock \doi{10.1137/S0040585X97976052}.

\bibitem[Goldys et~al.(2009)Goldys, Röckner, and Zhang]{GOLDYS2009}
Benjamin Goldys, Michael Röckner, and Xicheng Zhang.
\newblock Martingale solutions and markov selections for stochastic partial differential equations.
\newblock \emph{Stochastic Processes and their Applications}, 119\penalty0 (5):\penalty0 1725–1764, 2009.
\newblock ISSN 0304-4149.
\newblock \doi{10.1016/j.spa.2008.08.009}.

\bibitem[Adams and Fournier(2003)]{adams2003sobolev}
R.A. Adams and J.J.F. Fournier.
\newblock \emph{Sobolev Spaces}.
\newblock Pure and Applied Mathematics. Elsevier, 2003.
\newblock ISBN 9780080541297.

\bibitem[Willard(2004)]{willard2004general}
S.~Willard.
\newblock \emph{General Topology}.
\newblock Addison-Wesley series in mathematics. Dover Publications, 2004.
\newblock ISBN 9780486434797.

\bibitem[Schechter(1997)]{SCHECHTER1997}
Eric Schechter.
\newblock \emph{Handbook of Analysis and Its Foundations}.
\newblock Academic Press, San Diego, 1997.
\newblock ISBN 978-0-12-622760-4.
\newblock \doi{10.1016/B978-0-12-622760-4.X5000-6}.

\bibitem[Kelley(1975)]{kelley1975}
John~L. Kelley.
\newblock \emph{General Topology}.
\newblock Graduate Texts in Mathematics. Springer New York, NY, 1975.
\newblock ISBN 978-0-387-90125-1.

\bibitem[Aliprantis(2006)]{Aliprantis2006}
Aliprantis.
\newblock \emph{Topology}, pages 21--67.
\newblock Springer Berlin Heidelberg, Berlin, Heidelberg, 2006.
\newblock ISBN 978-3-540-29587-7.
\newblock \doi{10.1007/3-540-29587-9_2}.

\end{thebibliography}

\end{document}